\documentclass[a4paper,12pt]{scrartcl}

\usepackage[utf8]{inputenc}

\usepackage{amsthm,amsmath,amsfonts,amssymb}
\numberwithin{equation}{section}  
\usepackage{mathrsfs}

\usepackage{tikz-cd}

\usepackage{mathtools}	 
\usepackage{bm}	 
\usepackage{bbm}	 

\usepackage{tensor}	 

\usepackage[colorlinks=true]{hyperref}	 
\usepackage[all]{hypcap}

\newcommand{\m}[1]{\mathcal{#1}}
\newcommand{\bb}[1]{\mathbb{#1}}

\newcommand{\mrm}[1]{\mathrm{#1}}

\newcommand{\f}[1]{\mathfrak{#1}}

\newcommand{\diff}{\partial}
\newcommand{\bdiff}{\bar{\partial}} 

\newcommand{\I}{\mathrm{i}\mkern1mu}	

\DeclarePairedDelimiter\abs{\lvert}{\rvert}	
\DeclarePairedDelimiter\norm{\lVert}{\rVert}	

\DeclarePairedDelimiter{\set}{\{}{\}}	 
\newcommand{\tc}{\mathrel{}\mathclose{}\middle|\mathopen{}\mathrel{}}

\newcommand{\C}{\mathbb{C}}
\newcommand{\R}{\mathbb{R}}
\newcommand{\ii}{\operatorname{i}}
\newcommand{\opZ}{\operatorname{Z}}
\newcommand{\ch}{\operatorname{ch}}
\newcommand{\td}{\operatorname{Td}}

\newcommand{\dv}{\operatorname{div}}
\newcommand{\Rea}{\operatorname{Re}}
\newcommand{\Imm}{\operatorname{Im}}
\newcommand{\Ric}{\operatorname{Ric}}
\newcommand{\del}{\partial}
\newcommand{\delbar}{\bar{\partial}}
\newcommand{\PP}{\mathbb{P}}

\newtheorem{thm}{Theorem}[section]
 
\newtheorem{prop}[thm]{Proposition}
\newtheorem{lemma}[thm]{Lemma}
\newtheorem{cor}[thm]{Corollary}
\newtheorem{conj}[thm]{Conjecture}

\theoremstyle{definition}
\newtheorem{definition}[thm]{Definition}
\newtheorem{question}[]{Question}

\theoremstyle{remark}
\newtheorem{exm}[thm]{Example}
\newtheorem{rmk}[thm]{Remark}

\title{Special representatives of complexified K\"ahler classes}
\author{Carlo Scarpa and Jacopo Stoppa}
 
\date{ }

\begin{document}

\maketitle

\begin{abstract} Motivated by constructions appearing in mirror symmetry, we study special representatives of complexified K\"ahler classes, which extend the notions of constant scalar curvature and extremal representatives for usual K\"ahler classes. In particular, we provide a moment map interpretation, discuss a possible correspondence with compactified Landau-Ginzburg models, and prove existence results for such special complexified K\"ahler forms and their large volume limits in certain toric cases. 
\end{abstract}

\section{Motivation and results}

Let $X$ be a K\"ahler manifold of dimension $n$. We will always assume that $X$ is compact. Following the conventions of \cite{Sheridan_versality}, we define a \emph{complexified K\"ahler form} on $X$ as
\begin{equation*}
\omega^{\C} = \ii \omega + B \in \mathcal{A}^{1, 1}(X, \C),
\end{equation*}
where $\omega$ is a K\"ahler form and $B$ (the \emph{B-field}) is a real closed form. The corresponding class
\begin{equation*}
[\omega^{\C}] = [\ii \omega + B] \in H^{1, 1}(X, \C)
\end{equation*}
is called a \emph{complexified K\"ahler class}.

In this paper we introduce and study a \emph{complex} partial differential equation which attempts to fix a special representative for a complexified K\"ahler class $[\omega^{\C}]$. 
\begin{definition}
The \emph{constant scalar curvature K\"ahler (cscK) equation with $B$-field} is the complex PDE 
\begin{equation}\label{modelEqIntro}
s(\omega) + \gamma \frac{(\omega^{\C})^n}{\omega^n} = c,
\end{equation}
where $s(\omega) = \Lambda_{\omega} \Ric(\omega)$ denotes the scalar curvature, $\gamma \in \C$ is the \emph{complex coupling constant}, and $c \in \R$ is a \emph{real} topological constant, uniquely determined by $[\omega^{\C}]$ and $\gamma$.

More generally, let $\f{h}_0(\omega)$ denote the Lie algebra of real functions $\xi$ which are holomorphy potentials with respect to $\omega$, i.e. such that $\nabla^{1,0}_{\omega}\xi$ is a holomorphic vector field. Then, the \emph{extremal scalar curvature equation with $B$-field} is the complex PDE
\begin{equation}\label{extrEqIntro}
s(\omega) + \gamma \frac{(\omega^{\C})^n}{\omega^n} = \xi \in \f{h}_0(\omega).
\end{equation}
\end{definition}
These notions make sense for any K\"ahler manifold with a complexified polarisation, i.e. a pair $(X, [\omega^{\C}])$, and reduce to the cscK and K\"ahler extremal scalar curvature equations when the $B$-field vanishes. In this Section we motivate and discuss these equations in detail and state our main  results (see \ref{ResultsSec}), as well as several other useful properties.

\subsection{Mirror symmetry}\label{MirrorSec}
Pairs $(X, [\omega^{\C}])$ appear in the theory of mirror symmetry. When $X$ is Calabi-Yau, the open string $A$-model of $X$ is defined as (a suitable enhancement of the derived category of) the Fukaya $A_{\infty}$ category $\operatorname{Fuk}(X, \omega^{\C})$. The $B$-field enters crucially in the definition of $\operatorname{Fuk}(X, \omega^{\C})$: for example, objects are Lagrangians $L \subset X$, with respect to $\omega$, endowed with a unitary connection, with curvature $\operatorname{id}\otimes B|_{L}$. Similarly, morphisms are defined in terms of holomorphic discs and the monodromy of the unitary connections, twisted by the $B$-field (see \cite[Section $4.2.4$]{Sheridan_versality}). 

The resulting $A_{\infty}$ category $\operatorname{Fuk}(X, \omega^{\C})$ is independent of the choice of representative of $[\omega^{\C}]$, up to $A_{\infty}$-equivalence. But, by definition, in order to have a concrete model of the Fukaya category, and to distinguish objects with special geometric properties, for example stationary Lagrangians, one must fix a representative of the (complexified) K\"ahler class $[\omega^{\C}]$ (see \cite{SchoenWolfson, ThomasYau}; see also \cite{BiquardRollin_smoothing} for results on stationary Lagrangians in cscK manifolds). In the compact Calabi-Yau case, this is usually done by the Calabi-Yau theorem, solving the equations
\begin{align}\label{CYequationsIntro}
\operatorname{Ric}(\omega) = 0,\, \Delta_{\omega} B = 0.
\end{align} 
Beyond the compact Calabi-Yau case, one can formulate mirror symmetry similarly for Fanos (see e.g. \cite{KatzKontPant}), varieties of general type (see \cite{GrossKatzRuddat}), or for open Calabi-Yaus. 
 
In the Fano case, the analogue of \eqref{CYequationsIntro} is given by the equations
\begin{align}\label{FanoEquationsIntro}
\operatorname{Ric}(\omega) = \omega,\, \Delta_{\omega} B = 0,
\end{align} 
which, as is well known, are obstructed by $K$-stability. Even more importantly, these equations only make sense for the anticanonical polarisation. But according to \cite{KatzKontPant} (see also \cite[Section 0.5.3]{GrossHackingKeel_LCY} for the case of del Pezzo surfaces), the Fukaya category $\operatorname{Fuk}(X, \omega^{\C})$ of a compact Fano $X$ endowed with an \emph{arbitrary} complexified K\"ahler form $\omega^{\C}$ appears naturally in mirror symmetry, as equivalent to the category of matrix factorisations of a Landau-Ginzburg model $(Y, w)$ (a pair of a complex manifold $Y$ with a nonconstant holomorphic function $w$). The choice of complexified K\"ahler class is mirror to the choice of complex structure on $(Y, w)$, including the potential $w$. We are thus led to considering the problem of fixing \emph{special representatives for arbitrary complexified K\"ahler classes on a Fano} (not just for the real, anticanonical polarisation).  
 
Let us also look at the open case, for log Calabi-Yaus belonging to the important class of (interiors of) Looijenga pairs, appearing in the Gross-Siebert programme. We consider the two-dimensional case, for simplicity, see \cite{GrossHackingKeel_LCY} (and see e.g. \cite[Section 0.4]{GrossHackingKeel_LCY} for higher dimensions). Then, $(Y, D)$ is a pair given by a (necessarily rational) smooth projective surface $Y$, and singular, nodal, anticanonical curve $D \subset Y$. The complement $U = Y \setminus D$ is noncompact Calabi-Yau: it is endowed with the holomorphic symplectic form $\Omega = s^{-1}_D$, where $s_D$ is a defining section of $D$. Note that the simplest example is given by the case when $Y$ is toric, with a fixed toric structure, with toric boundary $D$.

According to \cite[Conjecture $0.9$]{GrossHackingKeel_LCY}, the relevant Fukaya category, \emph{in the case when $U$ is affine}, is the wrapped Fukaya category $\operatorname{Fuk}_{\operatorname{wr}}(U, \omega^{\C}|_{U})$, where $\omega^{\C}|_{U}$ is the \emph{restriction to $U$} of a complexified K\"ahler form $\omega^{\C}$ \emph{defined on the compactification} $Y$. Thus, in order to fix a representative of the $A_{\infty}$-equivalence class of $\operatorname{Fuk}_{\operatorname{wr}}(U, \omega^{\C}|_{U})$, we are led to considering special representatives $\omega^{\C}$ \emph{on the compactification $Y$}. The $B$-field should be part of the definition of such special representatives in a crucial way. (This is different of course from the important and difficult problem of finding complete Ricci flat metrics on the complement $U = Y \setminus D$). Note that the choice of a complexified K\"ahler class $[\omega^{\C}]$ corresponds to the choice of complex structure on the mirror affine log Calabi-Yau $\check{U}$, through an especially simple mirror map (this is a case when mirror symmetry holds globally, even far away from the large volume limit). A similar picture should hold in the higher dimensional case, at least when $D$ supports an ample divisor, see e.g. \cite[Conjecture 0.8]{GrossHackingKeel_LCY}. 

We will offer additional motivation for considering \eqref{modelEqIntro} as a way to fix a special representative of the complexified K\"ahler class $[\omega^{\C}]$. We show that \eqref{CYequationsIntro} and \eqref{FanoEquationsIntro} can be recovered as a (suitably normalised) \emph{large volume limit} of \eqref{modelEqIntro}; we discuss possible connections to Landau-Ginzburg models in \ref{MirrorConjectures}, and describe objects in Fukaya categories associated with solutions in \ref{MeanCurvSec}. 

In particular, in \ref{MirrorConjectures} below, we point out the general question of ``uniformising" mirror pairs in the Fano case, 
\begin{equation*} 
(X, [\omega^{\C}_X], s_X \in H^0(K^{-1}_X))\,|\,((Y, w), [\omega^{\C}_Y], \Omega_Y\in H^0(K_Y)),
\end{equation*}
i.e. of the relation between (the existence of) special representatives for the complexified K\"ahler classes $[\omega^{\C}_X]$ and $[\omega^{\C}_Z]$, where $f\!: Z \to \PP^1$ is a suitable (tame) compactification of the Landau-Ginzurg model $w\!: Y \to \C$, with $[\omega^{\C}_Z]|_Y = [\omega^{\C}_Y]$. For example, from a differential geometric perspective, we ask if the solvability of the Dervan-Ross equation (see \cite{DervanRoss_stablemaps}),
\begin{equation*} 
s(\omega_Z) - \Lambda_{\omega_Z}f^*\eta = c_Z, 
\end{equation*}
in a real K\"ahler class $[\omega_Z]$ on the compactified Landau-Ginzurg model $Z$ (for a fixed K\"ahler form $\eta$ on $\PP^1$), is related to the solvability of our equation
\begin{equation*}
s(\omega_X) + \gamma \frac{(\omega^{\C}_X)^n}{(\omega_X)^n} = c_X 
\end{equation*} 
on the mirror Fano manifold $(X, [\omega^{\C}_X])$, at least nearby certain limit points in the space of complex structures on $X$. 

Similarly, from an algebro-geometric perspective, in \ref{ComplexKstab} we ask for a characterisation of those Landau-Ginzburg models $((Y, w), [\omega^{\C}_Y], \Omega_Y)$ which are mirror to $K$-stable Fanos $(X, [\omega^{\C}_X], D_X = \dv(s_X))$, after a suitable extension of $K$-stability to complexified K\"ahler classes.
\subsection{Deformed Hermitian Yang-Mills connections}
Let us write \eqref{modelEqIntro}  in the form
\begin{equation*} 
s(\omega) + (-\ii)^n \bar{\gamma} \frac{(\omega + \ii B)^n}{\omega^n} = c \in \R.
\end{equation*}
Setting
\begin{equation*}
(-\ii)^n\bar{\gamma} = -|\gamma| e^{-\ii \hat{\theta}},
\end{equation*}
the angle $\hat{\theta}$ is uniquely determined, modulo $2\pi$, by the necessary reality condition
\begin{equation}\label{realityCond}
\int_X (\omega + \ii B)^n \in e^{\ii \hat{\theta}}\R_{> 0}.
\end{equation}
So, the single complex equation \eqref{modelEqIntro} is equivalent to the system
\begin{equation}\label{dKYM}
\begin{dcases}
  \Imm e^{-\ii \hat{\theta}} (\omega + \ii B)^n = 0\\
 s(\omega) - |\gamma| \Rea e^{-\ii \hat{\theta}} \frac{(\omega + \ii B)^n}{\omega^n} = c, 
\end{dcases}
\end{equation}
to be solved for the $(1,1)$ forms $\omega$, $B$ within their fixed cohomology classes. This is the system studied by Schlitzer and the second author in \cite{SchlitzerStoppa}. 

We see that the imaginary part of our complex PDE \eqref{modelEqIntro} is the well-studied \emph{deformed Hermitian Yang-Mills (dHYM) equation} appearing in the theory of mirror symmetry, see e.g. \cite{CollinsJacobYau, JacobYau_special_Lag, LeungYauZaslow}. This equation, when solvable (see \cite{GaoChen_Jeq_dHYM, Takahashi_dHYM, DatarPingali_dHYM, Lin_dHYM}), gives a unique way, compatible with mirror symmetry and expressing a vanishing moment map condition, to fix the $B$-field, for each choice of K\"ahler form $\omega$. 

More generally, fixing a holomorphic line bundle $L\to X$, we can follow ideas of Collins and Yau on the dHYM equation with a $B$-field (see \cite[p. $82$]{CollinsYau_momentmaps_preprint}) and consider the analogue of \eqref{modelEqIntro} given by
\begin{equation}\label{generalLIntro}
s(\omega) + \gamma \frac{(\ii \omega + B_{\omega} + F_L)^n}{\omega^n} = c,
\end{equation}
where $F_L = \ii F(A_{h_L})$ denotes the \emph{real} curvature of a Hermitian metric $h_L$ on the fibres of $L$, and $B_{\omega}$ is the $\omega$-harmonic representative of the class $[B]$. Note that this reduces to \eqref{modelEqIntro} in the special case $L \cong \mathcal{O}_X$. In the general case, we may regard $B_{\omega} + F_L$ as the special representative of the $B$-field class $[B + c_1(L)]$, and $\omega^{\C} = \ii \omega + B_{\omega} + F_L$ as the corresponding special representative of the complexified K\"ahler class $[\ii \omega + B + c_1(L)]$. Recall that adding to $B$ a closed, integral $(1,1)$-form changes the Fukaya category by an equivalence of $A_{\infty}$ categories (see \cite[Remark $4.11$]{Sheridan_versality}), so \eqref{generalLIntro} is still trying to fix a special complexified K\"ahler form yielding the same Fukaya category, up to equivalence.
\begin{rmk} The natural automorphism group for \eqref{dKYM} is the group $\operatorname{Aut}(X, [\omega^{\C}])$ of holomorphic automorphisms of $X$ fixing the cohomology classes $[\omega]$ and $[B]$. This is very different from the case of the twisted cscK equation (see e.g. \cite{StoppaTwisted}), where the twisting form $\alpha$ is fixed, and the automorphism group automatically vanishes as soon as $\alpha$ is positive. 
\end{rmk}
\subsection{Moment maps and Futaki invariant}
Our equation \eqref{modelEqIntro} (or \eqref{generalLIntro}) can also be motivated from moment map geometry. 
\begin{lemma}\label{momentMapLem} Suppose that the class of the $B$-field is Hodge. Then, the cscK equation with $B$-field \eqref{generalLIntro} corresponds to the vanishing moment map condition for a (formally complexified) Hamiltonian group action, with respect to a K\"ahler form determined by $|\gamma|>0$. 
\end{lemma}
\begin{proof}
Indeed, suppose that class of the $B$-field is Hodge, $[B] \in H^{1, 1}(X, \mathbb{Q})$. Then, we have $m [B] = c_1(N)$ for some holomorphic line bundle $N$ and some minimal $m \in \mathbb{N}$, and we observe
\begin{equation}\label{momentMapIntro}
s(\omega) + \gamma \frac{(\ii \omega + B_{\omega} + F_L)^n}{\omega^n} = c \iff s(m \omega) + \frac{\gamma}{m} \frac{(\ii m\omega + F_{N\otimes L^{m}})^n}{(m\omega)^n} = \frac{c}{m}.   
\end{equation}
In the latter equation the metric on the fibres of $N\otimes L^m$ is given by $h_{N} \otimes h^{\otimes m}_{L}$, where $F_{h_{N}} = m B_{\omega}$. According to \cite[Corollary 3]{SchlitzerStoppa}, the right hand side of \eqref{momentMapIntro} is precisely the vanishing moment map condition for a (formally complexified) Hamiltonian action of the extended gauge group $\widetilde{\mathcal{G}}$ on the product space $\mathcal{J}^{\operatorname{int}} \times \mathcal{A}$ of $(m\omega_0)$-compatible, integrable almost complex structures $\mathcal{J}^{\operatorname{int}}$ and connections $\mathcal{A}$ on $N\otimes L^m$, endowed with a suitable K\"ahler structure, determined by $|\gamma|$. 
\end{proof}
We will prove an analogue of this result also in the case of arbitrary $B$-field classes.
\begin{thm}\label{momentMapThm} The cscK equation with $B$-field \eqref{modelEqIntro}, when $[B] \in H^{1,1}(X, \R)$ is not necessarily Hodge, corresponds to the vanishing moment map condition for a (complexified) \emph{infinitesimal} Hamiltonian action (i.e., a Hamiltonian Lie algebra action), with respect to a K\"ahler form determined by $|\gamma|>0$.
\end{thm}
The proof is rather more technical and is given in Section~\ref{momentMapSec}. Although the infinitesimal Hamiltonian action is not a group action in general, we show in Section~\ref{momentMapSec} that the (complexified) infinitesimal action can be integrated; in other words, there is a meaningful notion of \emph{orbits} for this infinitesimal action.

Fix any representative~$\omega$, and let~$B = B(\omega)$ denote the corresponding dHYM representative for~$[B]$ with respect to~$\omega$, that is, the unique solution $B(\omega)$ to 
\begin{equation*} 
  \Imm e^{-\ii \hat{\theta}} (\omega + \ii B)^n = 0 
\end{equation*}
(assuming one exists; according to~\cite{GaoChen_Jeq_dHYM, Takahashi_dHYM, DatarPingali_dHYM}, this does not depend on the choice of representative for~$[\omega]$). Let~$\f{h}_0$ denote the Lie algebra of~$\omega$-Hamiltonian holomorphic vector fields on~$X$, i.e. vector fields~$V$ admitting a holomorphy potential~$\varphi(V, \omega)$ with respect to~$\omega$ (as well known,~$\f{h}_0$ does not depend on the choice of K\"ahler metric). 
\begin{definition}\label{def:Futaki_dHYM=0}
The Futaki invariant for~\eqref{modelEqIntro} is the linear function on~$\f{h}_0$ given by
\begin{equation*}
\m{F}_{[\omega^{\C}]}(V) = \int_X\varphi(V,\omega)\left(s(\omega)-|\gamma|\frac{\Rea\left(\mrm{e}^{-\ii\hat{\vartheta}}\left(\omega+\ii B(\omega)\right)^n\right)}{\omega^n}-c\right)\omega^n.
\end{equation*} 
Note that~$\m{F}_{[\omega^{\C}]} \in \f{h}^{\vee}_0$ is also a function of the real coupling~$|\gamma|$; we write~$\m{F}_{[\omega^{\C}], |\gamma|}$ when we need to emphasise this. Note that we have, by definition, 
\begin{equation*}
\m{F}_{[\omega^{\C}]} = \m{F}_{[\omega]} + |\gamma| \m{F}'_{[\omega^{\C}]}
\end{equation*}
where~$\m{F}_{[\omega]}$ is the classical Futaki character and the linear map~$\m{F}'_{[\omega^{\C}]} \in \f{h}^{\vee}_0$ does not depend on~$|\gamma|$.  
\end{definition} 
Standard moment map arguments, using Theorem~\ref{momentMapThm}, show that~$\m{F}_{[\omega^{\C}]} \in \f{h}^{\vee}_0$ (and so~$ \m{F}'_{[\omega^{\C}]}$) does not depend on the choice of representative for~$[\omega]$ (see Section~\ref{FutakiSec}). So we have the usual consequences, 
\begin{lemma} If the cscK equation with~$B$-field~\eqref{modelEqIntro} is solvable in~$[\omega^{\C}]$, then~$\m{F}_{[\omega^{\C}]} \equiv 0$. More generally, suppose~$\omega^{\C}$ solves the extremal equation with~$B$-field~\eqref{extrEqIntro}, namely,
\begin{equation*}
s(\omega) + \gamma \frac{(\omega^{\C})^n}{\omega^n} = \xi \in \f{h}_0(\omega),
\end{equation*}
where~$\f{h}_0(\omega)$ is the space of smooth real functions which are holomorphy potentials with respect to~$\omega$. Then, for any fixed choice of a maximal compact subgroup~$K \subset \operatorname{Aut}(X)$, there exists a unique~$V \in \operatorname{Lie}(K)$ (the \emph{extremal field}), identified with the Futaki invariant~$\m{F}_{[\omega^{\C}]} \in \f{h}^{\vee}_0$ under the Futaki-Mabuchi inner product, such that
\begin{equation*}
\nabla^{1, 0}_{\omega} \xi = V.
\end{equation*} 
\end{lemma} 
Thus, as usual, the possible extremal fields for our extremal equation~\eqref{extrEqIntro} are determined a priori, up to conjugation, by the Futaki invariant~$\m{F}_{[\omega^{\C}]}$. 
\subsection{Large volume limit and K\"ahler-Yang-Mills metrics}\label{KYMSec}
A crucial notion in mirror symmetry is that of \emph{large volume limit} (see e.g. \cite[Section 4]{Sheridan_versality}). For us, it means that we should analyse the formal behaviour of our equation under the scaling 
\begin{equation*}
\omega \mapsto k \omega,\,k\to \infty,\,\text{for } [\omega^{\C}] = [\ii \omega + B]\textrm{ fixed.}
\end{equation*}
Since $s(k \omega) = k^{-1} s(\omega)$, our equation \eqref{modelEqIntro} for 
\begin{equation*}
\omega^{\C}_k = B + \ii k \omega 
\end{equation*} 
is equivalent to 
\begin{equation*}
s(\omega) -  k|\gamma_k| e^{-\ii \hat{\theta}_k}\frac{(\omega + \ii k^{-1}B)^n}{\omega^n} = k c_k.
\end{equation*}
Here, $\gamma_k$ is a complex coupling constant, possibly depending on $k$, and $c_k$ is the corresponding real topological constant. Let us set 
\begin{equation}\label{zSlope}
z = n \frac{[\omega]^{n-1} \cup [B]}{[\omega]^n}.
\end{equation}
According to \cite[Proposition 7]{SchlitzerStoppa}, there are expansions
\begin{align}\label{LargeVolumeExpansion}
&\nonumber \Imm e^{-\ii \hat{\theta}_k}\frac{(\omega + \ii k^{-1}B)^n}{\omega^n} = k^{-1} (\Lambda_{\omega} B - z) + O(k^{-3}),\\
& \Rea e^{-\ii \hat{\theta}_k}\frac{(\omega + \ii k^{-1}B)^n}{\omega^n} = 1 - k^{-2} \left(\Lambda^2_{\omega}(B\wedge B) -  z \Lambda_{\omega} B +\frac{1}{2} z^2\right) + O(k^{-4}).  
\end{align}
So our sequence solves equations of the form
\begin{equation}\label{LargeVolumeEquations}
\begin{dcases}
-  |\gamma_k| (\Lambda_{\omega} B - z) + O(k^{-2}|\gamma_k|) = 0\\
s(\omega) + (k^{-1} |\gamma_k|)(\Lambda^2_{\omega}(B\wedge B) - z \Lambda_{\omega} B) + O(k^{-3}|\gamma_k|) = \hat{c}_k. 
\end{dcases}
\end{equation}
At a purely formal level, the limiting behaviour depends on 
\begin{equation*}
\lim_{k\to \infty} k^{-1} |\gamma_k| = \hat{\gamma} \in \R_{\geq 0} \cup \{\infty\}. 
\end{equation*}
The limiting equations become a special case of the K\"ahler-Yang-Mills system introduced by \'Alvarez-C\'onsul, Garcia-Fernandez and Garc\'ia-Prada \cite{AlvarezGarciaGarcia_KYM},  
\begin{equation}\label{KYMIntro}
\begin{dcases}
\Lambda_{\omega} B = z\\
s(\omega) + \hat{\gamma} (\Lambda^2_{\omega}(B\wedge B) - z \Lambda_{\omega} B )  = \hat{c}.
\end{dcases}
\end{equation}
Note that when $\hat{\gamma} = 0$, for example if the coupling $\gamma_k = \gamma$ is held constant, the limit decouples to
\begin{equation}\label{DecoupledHYM}
\begin{dcases}
\Lambda_{\omega} B = z \iff \Delta_{\omega} B = 0\\
s(\omega) = \hat{s},
\end{dcases}
\end{equation}
(where the stated equivalence follows at once from the K\"ahler identity $[\Lambda_{\omega}, \delbar] = \ii \delbar^*$ since $B$ is closed), and so in particular, in the Calabi-Yau case, the large volume limit of our equations becomes
\begin{equation*}
\begin{dcases}
\Delta_{\omega} B = 0\\
\Ric(\omega) = 0,
\end{dcases}
\end{equation*}
which are usually taken as the uniformising equations for Calabi-Yaus both in mathematics and physics references (where the parameter $k^{-1}$ appears as the ``string length constant" $\ell^2_s = k^{-1}$, see e.g. \cite[p. $4$]{DirichletBook}).

The extremal case \eqref{extrEqIntro} is entirely analogous, and the large volume limit is given by the extremal K\"ahler-Yang-Mills equations of \cite[Section 4.5]{AlvarezGarciaGarcia_KYM},
\begin{equation}\label{extrKYMIntro}
\begin{dcases}
\Lambda_{\omega} B = z\\
s(\omega) + \hat{\gamma} (\Lambda^2_{\omega}(B\wedge B) - z \Lambda_{\omega} B )  = \hat{\xi} \in \f{h}_0(\omega).
\end{dcases}
\end{equation} 
\subsection{Calabi-Volume functional}
Recall the \emph{complexifed volume functional}, introduced by Jacob and Yau \cite{JacobYau_special_Lag}, given by
\begin{equation*}
V_{\omega}(B) = \int_{X} r_{\omega}(B)\omega^n,
\end{equation*}
where the \emph{radius function} is defined as
\begin{equation*}
r_{\omega}(B) = \left| \frac{(\omega^{\C})^n}{\omega^n}\right|.
\end{equation*}
According to \cite{JacobYau_special_Lag}, solutions $B$ of the dHYM equation are minimisers of $V_{\omega}$, attaining the minimum 
\begin{equation*}
0 < V_{\omega}(B) = \left| \int_{X}  (\omega + \ii B)^n\right|.
\end{equation*} 
\begin{definition} We introduce a \emph{Calabi-Volume functional} given by
\begin{equation*}
\operatorname{CVol}(\omega^{\C}) = \int_X \left( s(\omega) - |\gamma| r_{\omega}(B) \right)^2  \omega^n  + |\gamma| V_{\omega}(B).
\end{equation*}
\end{definition}
\begin{lemma} Solutions of the cscK equation with $B$-field \eqref{modelEqIntro}, such that $c < \frac{1}{2}$, minimise the Calabi-Volume functional. 
\end{lemma}
\begin{proof}
By a straightforward computation, we have
\begin{align*}
\operatorname{CVol}(\omega^{\C}) &= \int_X \left( s(\omega) - |\gamma| r_{\omega}(B) - c\right)^2  \omega^n + |\gamma| V_{\omega}(B) \\
& + 2 c \int_X\left( s(\omega) - |\gamma| r_{\omega}(B)\right) \omega^n - c^2 \int_X \omega^n\\
&= \int_X \left( s(\omega) - |\gamma| r_{\omega}(B) - c\right)^2  \omega^n + (1-2c) |\gamma| V_{\omega}(B) +  c (2\hat{s} - c) [\omega]^n.
\end{align*}
Recall that \eqref{modelEqIntro} is equivalent to \eqref{dKYM}. The claim then follows from the fact that dHYM solutions are minimisers for $V_{\omega}(B)$, and that, at a dHYM solution, we have
\begin{equation*}
r_{\omega}(B) = \Rea e^{-\ii \hat{\theta}} \frac{(\omega + \ii B)^n}{\omega^n}.\qedhere
\end{equation*}   
\end{proof}
\begin{rmk} Suppose that $B = B(\omega)$ is a dHYM solution corresponding to $\omega$. Then, by \eqref{LargeVolumeExpansion}, there is a large volume limit expansion
\begin{equation*}
\operatorname{CVol}(\ii\omega + k^{-1}B) = \operatorname{CYM}(\omega, B) + O(k^{-1}),
\end{equation*}
where the Calabi-Yang-Mills functional was defined (in much greater generality) in \cite[Section 2.4]{GarciaFernandez_PHD} as
\begin{equation*}
\operatorname{CYM}(\omega, B) = \int_X \left(s(\omega) - \hat{\gamma} |B|^2_{\omega}\right) \omega^n + \hat{\gamma} \int_X |B|^2_{\omega}\,\omega^n.
\end{equation*}
\end{rmk}
\begin{rmk} The variational interpretation can be used more generally, by the trick of multiplying \eqref{modelEqIntro} by a suitably small real constant, see \cite[Remark 2.4.2]{GarciaFernandez_PHD}. 
\end{rmk}

\subsection{Existence results}\label{ResultsSec}
We can now state our existence results, which bear on the case when $X$ is toric. 

The toric case is relevant for mirror symmetry for Fano and log Calabi-Yau manifolds, see Section \ref{MirrorSec}. In \ref{MirrorConjectures} we discuss a possible interpretation of these results, in particular of the coupling constant $|\gamma|$, in terms of compactified Landau-Ginzburg models. Note that in the toric Fano case the mirror Landau-Ginzburg model can be described very explicitly, see e.g. \cite[Example 1.15]{GrossHackingKeel_LCY}.
\begin{thm}\label{MainThmIntro} Let $(X, [\omega])$ be a polarised K\"ahler toric manifold, with $\dim(X) \leq 3$, which is uniformly $K$-stable with constant $\lambda$. Then, for all sufficiently small $\epsilon > 0$, there is an open set of \emph{positive} (i.e. K\"ahler) $B$-field classes $\mathcal{B}_{\epsilon} \subset H^{1,1}(X, \R)$, containing a ball of radius $\kappa \epsilon \delta$ around $\kappa \epsilon [\omega]$ for some fixed constants $\kappa, \delta > 0$ (in the metric induced by $\omega$), such that, for $[B] \in \mathcal{B}_{\epsilon}$ and all coupling constants in the range  
\begin{equation*}
0 \leq |\gamma| < \frac{\hat{s}\,\lambda}{2(1-\lambda)} \epsilon^{-1},
\end{equation*}
the extremal scalar curvature equation with $B$-field \eqref{extrEqIntro} is solvable in the complexified K\"ahler class $[\omega^{\C}] = [\ii \omega + B]$.
\end{thm}
The main point is that we obtain quantitative control on the coupling constant $|\gamma|$. The limitation to surfaces and threefolds is due to special properties of the dHYM equation. 
\begin{cor}\label{MainCorIntro} Fix $0 < \gamma' < 1$. For all sufficiently large $k > 0$, choose $[B_k] \in \mathcal{B}_{k^{-1}}$ and let $\omega^{\C}_k = \ii \omega_k + B_k$ denote the solution of \eqref{modelEqIntro} with real coupling constant
\begin{equation*}
|\gamma_k| =  \frac{\hat{s}\,\lambda}{2(1-\lambda)} \gamma' k, 
\end{equation*}   
provided by Theorem \ref{MainThmIntro}. Then, the rescaled complexified form $\omega_k + \ii k B_k$ converges smoothly, in the large volume limit $k \to \infty$, to a solution of the extremal K\"ahler-Yang-Mills system \eqref{extrKYMIntro}, with coupling constant
\begin{equation*}
\hat{\gamma} = \frac{\hat{s}\,\lambda}{2(1-\lambda)}\gamma'. 
\end{equation*} 
\end{cor}
Theorem \ref{MainThmIntro} and Corollary \ref{MainCorIntro} are proved in Section \ref{SurfacesSec} (for surfaces) and Section \ref{ThreefoldsSec} (for threefolds), relying on estimates established in Section \ref{ToricEstimatesSec}, based on the theory developed by Chen and Cheng \cite{ChenCheng_estimates}. These results provide basic examples of extremal metrics with $B$-field, and of their large volume behaviour, on cscK toric manifolds. We note that, in particular, \emph{Corollary \ref{MainCorIntro} gives new examples of extremal K\"ahler-Yang-Mills metrics} (cf. \cite[Theorem 4.17]{AlvarezGarciaGarcia_KYM}).
\begin{rmk}
Note that uniform $K$-stability implies the existence of a cscK metric $\omega_0 \in [\omega]$ (see Section \ref{ToricEstimatesSec}). In particular, in the situation of Theorem \ref{MainThmIntro}, we have $\m{F}_{[\omega]}(0) = 0$. Among the extremal metrics with $B$-field (i.e. solutions of \eqref{extrEqIntro}) constructed in Theorem \ref{MainThmIntro}, standard arguments show that the locus of cscK metrics with $B$-field (solutions of \eqref{modelEqIntro}) is given by
\begin{equation*}
\mathcal{B}_{\epsilon}\,\cap \{\m{F}_{[\omega^{\C}], |\gamma|} = 0\} = \mathcal{B}_{\epsilon}\,\cap \{\m{F}'_{[\omega^{\C}]} = 0\}.  
\end{equation*}
However, at present, we do not have an effective characterisation of the locus $\{\m{F}'_{[\omega^{\C}]} = 0\}$. The same problem, in the large volume limit, appears in \cite[Section 4.5]{AlvarezGarciaGarcia_KYM}. For example, in the case when $X$ is a surface, one can show that, for a \emph{fixed} class $[\omega]$, the point $[B] = [\omega]$ is a critical point of the function 
\begin{equation*}
H^{1,1}(X, \R) \ni B \mapsto \m{F}'_{[\ii \omega + B]} \in \f{h}^{\vee}_0,
\end{equation*}
and that the Hessian at this critical point can be identified with the bilinear form on the space of $\omega$-harmonic forms $\mathcal{H}^{1,1}_{\omega}(X, \R)$ given by
\begin{equation*}
D^2(\m{F}'_{[\ii \omega + B]})|_{[B]=[\omega]} [\eta_1, \eta_2] (V) = \int_{X} f(V, \omega)\,\eta_1 \wedge \eta_2.  
\end{equation*}
\end{rmk}

For arbitrary dimension, we only have a much weaker result than Theorem \ref{MainThmIntro}.
\begin{prop}
The statement of Theorem \ref{MainThmIntro} remains valid, with $X$ toric of arbitrary dimension, for some $|\gamma| > 0$.
\end{prop}
This follows at once from the openness result discussed in Section \ref{opennessSec}.

We also show, in Section \ref{KUnstSec}, that the cscK equation with a $B$-field may be solvable in a $K$-unstable K\"ahler class on a toric, K\"ahler-Einstein Fano (for which a cscK representative does not exist). 
\begin{thm}\label{UnstableThmIntro}  Consider the toric Fano K\"ahler-Einstein threefold 
\begin{equation*}
X = \PP(\mathcal{O}\oplus\mathcal{O}(1,-1)) \to \PP^1 \times \PP^1.
\end{equation*}
Then, for all $|\gamma| > 0$, there exist $K$-unstable K\"ahler classes $[\omega]$ and $B$-field classes $\mathcal{B}_{[\omega], |\gamma|} \subset H^{1,1}(X, \R)$, such the the cscK equation with $B$-field \eqref{modelEqIntro} is solvable for $[\omega^{\C}] = [\ii \omega + B]$, $[B]\in \mathcal{B}_{[\omega], |\gamma|}$.  
\end{thm}

\begin{rmk}
The solutions of \eqref{extrEqIntro} on toric surfaces and threefolds given by Theorem \ref{MainThmIntro} are torus-invariant. The question of the \emph{uniqueness} of these solutions (in their classes) is still open. In a recent preprint \cite{Scarpa_uniqueness} the first-named author shows that, on toric manifolds of arbitrary dimension, the uniqueness of torus-invariant solutions of \eqref{modelEqIntro} holds under a \emph{supercritical phase condition}, but the classes we consider in Theorem \ref{MainThmIntro} do not satisfy this hypothesis for $\dim X=3$. In the two-dimensional case instead the results of \cite{Scarpa_uniqueness} apply directly. Hence, on a $K$-stable surface for which $\mathcal{F}'_{[\omega^{\mathbb{C}}]}=0$, the solutions given by Theorem \ref{MainThmIntro} are unique in their classes, up to the action of the reduced automorphism group of $X$.
\end{rmk}
 
\subsection{Conjectures} 

\subsubsection{Mirror pairs and $K$-stability}\label{MirrorConjectures}

Our main speculation here concerns the question of whether mirror symmetry for Fano pairs, after suitable compactification, can be made compatible with versions of $K$-stability and the corresponding partial differential equations, in particular the cscK equation with $B$-field \eqref{modelEqIntro}. As will be clear from the discussion, it is indeed necessary to allow a nontrivial $B$-field in general.

Let 
\begin{equation}\label{MirrorPair}
(X, [\omega^{\C}_X], s_X \in H^0(K^{-1}_X))\,|\,((Y, w), [\omega^{\C}_Y], \Omega_Y\in H^0(K_Y))
\end{equation}
be a mirror Fano pair as in \ref{MirrorSec} above (see \cite[Section 2.1]{KatzKontPant} for full details). Following \cite[Section 2.1]{KatzKontPant}, we consider a \emph{tame} compactification $f\!: Z \to \PP^1$ of the map $w\!: Y \to \C$, where $Z$ is a smooth projective variety containing $Y$ as the complement of a suitable divisor (see \cite[Definition 2.4]{KatzKontPant}). Such compactifications have been studied in depth. For example, if $X$ is a smooth Fano threefold, a Calabi-Yau compactified Landau-Ginzburg model exists \cite{Przyjalkowski_CYLG}, i.e., for suitable K\"ahler parameters on $X$, we can choose $f\!: Z \to \PP^1$ to be a Calabi-Yau fibration, with $-K_{Z} \sim f^{-1}(\infty)$. This compactification is only unique up to birational equivalence. 

Assume that the complexified K\"ahler class on $Y$ is in fact a real K\"ahler class $[\omega_Y]$.  
\begin{question}\label{LGquestion} Characterise mirror pairs \eqref{MirrorPair}, for which the Landau-Ginzburg model $((Y, w), [\omega_Y], \Omega_Y)$ has the following stability property: a tame compactification $(Z, f)$ exists, together with a K\"ahler class $[\omega_Z]$ restricting to $[\omega_Y]$, such that the polarised fibration
\begin{equation*}
f\!: (Z, [\omega_Z]) \to (\PP^1, [\eta])
\end{equation*}    
is (uniformly) $K$-stable (or log $K$-stable, with respect to the compactification divisor $D_Z \subset Z$), for some K\"ahler class $[\eta]$ (see \cite{DervanRoss_stablemaps}).   
\end{question}
Note that $K$-stability is invariant under overall scaling, i.e. $f\!: (Z, [\omega_Z]) \to (\PP^1, [\eta])$ is $K$-stable iff $f\!: (Z, k[\omega_Z]) \to (\PP^1, [\eta])$ is, for all $k > 0$. On the other hand, the mirror correspondence is not scale invariant, and  
\begin{equation*}
(X_k, [\omega^{\C}_{X_k}], s_{X_k})\,|\,((Y, w), k[\omega_Y], \Omega_Y),\,k \to \infty
\end{equation*}
gives a sequence of mirror pairs with $(X_k, s_{X_k})$ approaching a limit point in the space of complex structures.

Note also that geometric characterisations of (uniform, log) $K$-stability for Calabi-Yau fibrations are known, at least in the \emph{adiabatic limit} when the volume of the fibres, with respect to $[\omega_Z]$, is sufficiently small \cite{Hattori_CYfibrations}; these would apply to suitable polarisations on $(Z, f)$. 
\begin{question}\label{FanoQuestion} Let 
\begin{equation*}
(X, [\omega^{\C}_X], s_X)\,|\,((Y, w), [\omega_Y], \Omega_Y)
\end{equation*}
be a mirror Fano pair, such that the polarisation on $Y$ is real. Suppose $((Y, w), [\omega_Y], \Omega_Y)$ is stable (or log stable) in the sense of Question \ref{LGquestion}. Is it possible to solve the cscK (more generally, extremal) equation with $B$-field \eqref{modelEqIntro} on $(X, [\omega^{\C}_X])$ (respectively, with suitable cone angle along $D_X = \dv(s_X)$), for some range of the coupling constant $|\gamma|$? Is this true at least in more restrictive situations: for example, for Calabi-Yau compactifications, and for $(X_k, [\omega^{\C}_{X_k}], s_{X_k})$, with $k$ sufficiently large?  
\end{question} 
Recall that, according to \cite[Conjecture 1.6]{DervanRoss_stablemaps}, the (uniform) $K$-stability of the fibration $f\!: (Z, [\omega_Z]) \to (\PP^1, [\eta])$ is conjecturally equivalent to the solvability of the equation
\begin{equation}\label{RossDervanEqu}
s(\omega_Z) - \Lambda_{\omega_Z}f^*\eta = c_Z.
\end{equation}
So, from the differential-geometric viewpoint, we are asking for a possible relation between solutions of \eqref{RossDervanEqu} and of the equation
\begin{equation*}
s(\omega_X) + \gamma \frac{(\omega^{\C}_X)^n}{(\omega_X)^n} = c_X.
\end{equation*}
\begin{exm} Let $(X, [\omega^{\C}_X = \ii \omega_X + B_X], s_X)$ be a smooth del Pezzo surface endowed with a complexified K\"ahler class and an anticanonical section $s_X$ such that $D_X = \dv(s_X)$ is smooth. Then, under some assumptions, there exists a smooth rational elliptic surface $f\!: Z \to \PP^1$, with a K\"ahler class $[\omega_Z]$, such that the mirror Landau-Ginzburg model is given by
\begin{equation*}
\big(Y = Z \setminus f^{-1}(\infty), [\omega_Z]|_{Y}, \Omega_Y = s^{-1}_{f^{-1}(\infty)}\big),
\end{equation*}
see \cite[Sections 0.5.2, 0.5.3]{GrossHackingKeel_LCY}; \cite{CollinsJacobLin_SYZ} proves a large part of the Strominger-Yau-Zaslow conjecture in this case. Suppose we choose a sequence $(X_k, [\omega^{\C}_{X_k}], s_{X_k})$ so that the corresponding polarisation $k[\omega_{Z,k}]$, on the fixed elliptic fibration $f\!:Z\to\PP^1$, is real and approaches the \emph{rescaled} adiabatic limit, in which the volume of the fibres is fixed, while the volume of the base blows up. 

By \cite[Sections 0.5.2]{GrossHackingKeel_LCY}, we have a relation
\begin{equation*}
S' \ni j_k = \exp\big(-2\pi k [\omega_{Z,k}]\big),
\end{equation*}
where $S'$ is a partially compactified versal family of $X\setminus D_X$. As $k \to \infty$, the limit converges to a point of $S'$ corresponding to a singular del Pezzo, with Gorenstein SLC singularities.

On the other hand, according to \cite[Corollary H]{Hattori_CYfibrations}, the $K$-stability of the polarised elliptic fibration $f\!: (Z, [\omega_{Z,k}]) \to \PP^1$, nearby the adiabatic limit, is determined by the singularities of the fibres. In particular if all fibres are reduced then the fibration is $K$-stable for all sufficiently large $k$. 

So, if some correspondence as in Question \ref{FanoQuestion} holds, then we would obtain cscK metrics, \emph{in general with nonvanishing $B$-field}, on the del Pezzo surfaces $(X_k, [\omega^{\C}_{X_k}])$, for all large $k$, i.e. as $(X_k, [\omega^{\C}_{X_k}])$ acquires certain Gorenstein SLC singularities. 
\end{exm}
\begin{exm} Consider $(X, [\omega^{\C}_X = \ii \omega_X + B_X], s_X)$ with $X$ a smooth toric Fano threefold. Let $((Z, f), [\omega_Z])$ be a Calabi-Yau compactification of its Landau-Ginzburg model $((Y, w), [\omega_Y], \Omega_Y)$ (itself a partial compactification of Givental's map $w\!: (\C^*)^3 \to \C$, depending on $s_X$, see \cite[Section 2.1]{KatzKontPant}). Note that, since $X$ is rigid, varying the K\"ahler class $[\omega_Z]$ corresponds to varying the section $s_X$. Let us choose $[\omega_Z]$ sufficiently close to the adiabatic limit. It seems reasonable that some analogue of the results of \cite{Hattori_CYfibrations} for rational elliptic surfaces holds in this case, so that if the singularities of $f\!: Z \to \PP^1$ are sufficiently generic then it is $K$-stable with respect to $[\omega_Z]$. 

Through some correspondence as in Question \ref{FanoQuestion}, this predicts results such as Theorem \ref{UnstableThmIntro}: even if $(X, [\omega_X])$ is $K$-unstable, so the fibres of $f\!: Z \to \PP^1$ are too singular, allowing a sufficiently large $B$-field class, with respect to a fixed $|\gamma|>0$, corresponds to a large deformation of $f$, which will have sufficiently regular fibres.  

Note that \cite[Theorem 3, Corollary 9]{SchlitzerStoppa_examples} provide an analogue of Theorem \ref{UnstableThmIntro} for complexified K\"ahler classes on the toric del Pezzo surface $\mathbb{F}_1 \cong \PP_{\PP^1}(\mathcal{O}\oplus\mathcal{O}(1))$ (which is $K$-unstable with respect to \emph{all} $[\omega]$). The solutions of \eqref{modelEqIntro} constructed in that case have conical singularities along (components of) the toric boundary $D_{\mathbb{F}_1}$.

In particular, \cite[Theorem 3, Corollary 9]{SchlitzerStoppa_examples} shows that, for all $[\omega]$, the cscK equation with $B$-field \eqref{modelEqIntro} is solvable on $(\mathbb{F}_1, [\omega^{\C} = \ii \omega + B])$, for some sufficiently large $|\gamma|[B]$, with a maximal cone angle along $D_{\mathbb{F}_1}$ which \emph{only depends (explicitly) on $[\omega]$, not on $[B]$}. This seems compatible with the statement that the compactified elliptic fibration $f\!: Z \to \PP^1$, mirror to $(\mathbb{F}_1, [\omega], D_{\mathbb{F}_1})$ is log $K$-unstable, with respect to some maximal angle depending on $[\omega]$, but becomes stable for this same angle after a sufficiently large deformation of complex structure.

Theorem \ref{MainThmIntro} and Corollary \ref{MainCorIntro} would have similar interpretations on the Landau-Ginzburg side: they are consistent with the fact that a uniformly $K$-stable fibration $f\!: Z \to \PP^1$ remains stable, even as we approach a large complex structure limit point, under sufficiently small deformations of complex structure, parametrised by $|\gamma| \mathcal{B}_{\epsilon}$ and so depending on the uniform stability constant $\lambda$.  

\begin{rmk} The above examples suggest that the range of the coupling constant $|\gamma|$ should be related to the width of an annulus parametrising mirror deformation of complex structures with respect to suitable normalised canonical coordinates on the complex Landau-Ginzburg moduli space (see \cite[Section 3.3]{KatzKontPant}).  
\end{rmk}
\end{exm}

\subsubsection{Solutions and $K$-stability}\label{ComplexKstab}

It is also natural to expect that a version of the Yau-Tian-Donaldson conjecture can be formulated for the cscK equation with $B$-field. In fact, we show in Section \ref{SurfacesSec} that, when $X$ is a complex surface, after a suitable change of variables, our equation \eqref{modelEqIntro}, at least for a special value of the coupling, is equivalent to the coupled K\"ahler-Einstein equations \cite{HultgrenWitt_coupled}, or the coupled cscK equations \cite{DatarPingali_coupled} for a pair of K\"ahler metrics, for which a Yau-Tian-Donaldson correspondence has already been proposed.
\begin{conj} Let $X$ be a K\"ahler surface, with discrete automorphisms. The cscK equation with $B$-field \eqref{modelEqIntro} is solvable in the complexified K\"ahler class $[\omega^{\C}] = [\ii \omega + B]$, with coupling constant $|\gamma| = \sin(\hat{\theta})$, iff the pair $([\omega], [\chi := \sin(\hat{\theta}) B + \cos(\hat{\theta}) \omega])$ is uniformly $K$-stable in the sense discussed in \cite[Section 1]{HultgrenWitt_coupled} and \cite[Section 3.1]{DatarPingali_coupled}. 
\end{conj}
Here, $e^{\ii \hat{\theta}}$ is the topological angle, defined in \eqref{realityCond}, and we are assuming, without loss of generality, that we have $\sin(\hat{\theta}) > 0$, and that $[\chi]$ is a K\"ahler class (the condition that $\pm [\chi]$ is a K\"ahler class is necessary for \eqref{modelEqIntro} to be solvable, as it is equivalent to the solvability of the dHYM equation on a surface, see Section \ref{SurfacesSec} for more details).  

In this context, we can complement Question \ref{LGquestion} by asking, roughly speaking, what condition corresponds to $K$-stability under mirror symmetry. 
\begin{question} Is there an algebro-geometric characterisation for Landau-Ginzburg models $((Y, w), [\omega^{\C}_Y], \Omega_Y)$ which are mirror to a del Pezzo surface $(X, [\omega^{\C}_X], s_X)$ with a $K$-stable complexified polarisation $[\omega^{\C}_X]$ (i.e., such that the pair $([\omega_X], [\chi := \sin(\hat{\theta}) B_X + \cos(\hat{\theta}) \omega_X])$ is (uniformly, log) $K$-stable in the sense of \cite[Section 1]{HultgrenWitt_coupled} and \cite[Section 3.1]{DatarPingali_coupled})?  
\end{question}  
\subsubsection{Moduli spaces}
By analogy with the classical works \cite{Schumacher_CY, FujikiSchumacher_moduli} and the more recent results \cite{DervanNaumann_moduli}, we expect that our equation \eqref{modelEqIntro} can be used in order to construct moduli spaces of K\"ahler manifolds endowed with a complexified polarisation, i.e. pairs $(X, [\omega^{\C}])$.
\begin{conj} There exists a Hausdorff complex space $\mathfrak{M}_{\gamma}(X, [\omega^{\C}])$ which is a moduli space of K\"ahler manifolds with a complexified polarisation (in the same sense as \cite{FujikiSchumacher_moduli, DervanNaumann_moduli}), such that the cscK equation with $B$-field \eqref{modelEqIntro} is solvable in the class $[\omega^{\C}]$. Moreover, $\mathfrak{M}_{\gamma}(X, [\omega^{\C}])$ admits a natural Weil-Petersson type K\"ahler metric $\eta$.
\end{conj}
It should be possible to prove this using the approach of \cite{DervanNaumann_moduli} at least in the special case when $\operatorname{Aut}(X, [\omega]^{\C})$ is discrete, while the general case seems much harder. Note that by Theorem \ref{UnstableThmIntro} the moduli space of solutions of \eqref{modelEqIntro} could be nonempty even when the cscK moduli space is empty. 
\begin{exm}
Following Section \ref{KYMSec}, we consider the moduli space $\mathfrak{M}_{\gamma}(X, [\ii \omega + k^{-1} B])$ for fixed $\gamma$ and $\ii \omega + B$, nearby the large volume limit $k \to \infty$. 

Suppose $X$ is Calabi-Yau. Then, for all sufficiently large $k$, we should have   
\begin{equation*}
\mathfrak{M}_{\gamma}(X, [\ii \omega + k^{-1} B]) \cong \mathfrak{M}_B(X, [\omega]) \subset \mathfrak{M}(X, [\omega]),
\end{equation*}
the fixed subspace of the moduli space of polarised Calabi-Yaus $\mathfrak{M}(X, [\omega])$ where $[B]$ is of type $(1, 1)$. This is because the equations \eqref{LargeVolumeEquations}, for fixed $\gamma_k = \gamma$, should be uniquely solvable for all large $k$, by analogy with the results of \cite[Section 4]{AlvarezGarciaGarcia_KYM}. Then, the K\"ahler metrics $\eta_k$ on this fixed subspace $\mathfrak{M}_B(X, [\omega])$, induced from $\mathfrak{M}_{\gamma}(X, [\ii \omega + k^{-1} B])$, provide natural deformations of the (restriction of) the classical Weil-Petersson metric on the Calabi-Yau moduli space, parametrised by the ``string length" parameter $\ell^2_s = k^{-1}$.  
\end{exm}
 
\subsection{Generalisations}
\subsubsection{Critical connections and critical metrics}
Dervan, McCarthy and Sektnan \cite{Dervan_Zconnections} studied generalisations of the dHYM equation, for arbitrary rank, associated with a polynomial central charge in the sense of Bayer \cite{BayerPolynomial}. 
Our equation \eqref{modelEqIntro} admits a straightforward generalisation to this context, namely
\begin{equation}\label{generalZintro}
s(\omega) + \gamma \frac{\tilde{\opZ}([B])}{\omega^n} = c \in \R,
\end{equation}
where 
\begin{equation*}
\tilde{\opZ}([B]) = \tilde{\opZ}([B], L) = \tilde{\opZ}_{\omega, k, B_{\omega}}(L)
\end{equation*}
denotes the $(n, n)$-form corresponding to a choice of polynomial central charge $\opZ$, evaluated on the Hermitian line bundle $(L, h_L)$, with respect to the representatives $\omega$, $B_{\omega}$, the latter denoting the unique solution to $\Lambda_{\omega} B_{\omega} = b \in \R$, i.e., the unique $\omega$-harmonic representative. Then, the special representative of $[\omega^{\C}]$, modulo $H^{1,1}(X, \mathbb{Z})$, corresponding to a solution of \eqref{generalZintro} is given by
\begin{equation*}
\omega^{\C} = \ii \omega + B_{\omega} + F_{L}, 
\end{equation*}
where $F_L = F_{h_L} = \ii F(A_{h_L})$ is the curvature of the metric $h_L$ on the fibres of $L$. Thus, our model equation \eqref{modelEqIntro} corresponds to 
\begin{equation*}
s(\omega) + \gamma \frac{\tilde{\opZ}_{\textrm{dHYM}}(\mathcal{O}_X)}{\omega^n} = c,
\end{equation*}
for the central charge
\begin{equation*}
\opZ_{\textrm{dHYM}}(E) = - \int_X e^{-\ii \omega - B} \ch(E).
\end{equation*}
Another relevant choice is the PDE 
\begin{equation*}
s(\omega) + \gamma \frac{\tilde{\opZ}_{\textrm{Todd}}(\mathcal{O}_X)}{\omega^n} = c,
\end{equation*}
corresponding to the central charge
\begin{equation*}
\opZ_{\textrm{Todd}}(E) = - \int_X e^{-\ii \omega - B} \ch(E) \sqrt{\td(X)}.
\end{equation*}
In this case, when taking the representative $(n,n)$-form $\tilde{\opZ}$, one uses a \emph{fixed} representative of $\sqrt{\td(X)}$.

One can show that the moment map interpretation, Lemma \ref{momentMapLem}, extends to \eqref{generalZintro}, at least when the $B$-field class is Hodge.

Dervan \cite{Dervan_crit_metrics} also studies a notion of critical K\"ahler metrics $\omega$ which deforms the scalar curvature $s(\omega)$ (nearby the large volume limit) with a suitable notion of central charge $\operatorname{\tilde{Z}}(\omega)$. Accordingly, one may consider extending this to complexified K\"ahler classes by the equation
\begin{equation*}
\operatorname{\tilde{Z}}(\omega^{\C}) := \operatorname{\tilde{Z}}(\omega) + \gamma \operatorname{\tilde{Z}}([B]) = \zeta \in \C^*.    
\end{equation*}  
\subsubsection{Anticanonical divisors}
As discussed above, according to \cite{KatzKontPant}, to describe Fano mirror pairs we should consider triples $(X, [\omega^{\C}_X = \ii \omega + B], s_X)$, where $s_X \in H^0(K^{-1}_X)$ is an anticanonical section (so $\Omega_X = s^{-1}_X$ is a meromorphic volume form, with simple poles along $D_X = \dv(s_X)$). 

Under some assumptions, it is possible to write an analogue of \eqref{generalLIntro} for the full data $(X, [\omega^{\C}_X], s_X)$. As in the proof of Lemma \ref{momentMapLem}, fix line bundles $N, L$ on $X$ such that $m [B] = c_1(N)$, for some $m \in \mathbb{Z}$, and moreover $N \otimes L^{m} \cong K^{-1}_X$. Then, we consider the equation  
\begin{equation}\label{FanoEqu}
s(\omega) + \gamma \frac{(\ii\omega + B_{\omega} + F_L)^n}{\omega^n} + \eta\big(\Delta_{\omega} -\frac{\ii}{2}\big)  | s_X |^{2}_h = c + \ii \tau,
\end{equation}
to be solved for a K\"ahler form $\omega$, a metric $h$ on the fibres of $K^{-1}_X$ given by $h_{N} \otimes h^{\otimes m}_{L}$, where $F_{h_{N}} = m B_{\omega}$, and real constants $\eta, c, \tau$. Extending the argument of Lemma \ref{momentMapLem} by the moment map picture in \cite{AlvarezGarciaGarcia_KYMH} shows that \eqref{FanoEqu} is also a vanishing moment map condition for a Hamiltonian group action. The complexified K\"ahler class representative associated with a solution is $\omega^{\C} = \ii \omega + B_{\omega} + F_L$. The large volume limit is the given by the K\"ahler-Yang-Mills-Higgs system studied in \cite{AlvarezGarciaGarcia_KYMH}.\\
 
\noindent\textbf{Acknowledgements.} We are grateful to Rudha\'i Dervan, Mario Garc\'{\i}a-Fern\'{a}ndez, Julien Keller, Sohaib Khalid, John McCarthy, Annamaria Ortu, James Pascaleff, Enrico Schlitzer, Lars Sektnan, Nicol\`o Sibilla, Zak Sj\"ostr\"om Dyrefelt, Paolo Stellari, Richard Thomas for several helpful discussions related to the present paper. We also wish to thank the anonymous Referee for some corrections and useful suggestions.
 
\section{Surfaces}\label{SurfacesSec}
In this Section we discuss our equations on a K\"ahler surface. We explain a connection to coupled K\"ahler-Einstein and cscK metrics, introduce certain objects in the Fukaya category associated with solutions, and prove Theorem \ref{MainThmIntro} in the case of surfaces.
\subsection{The equations on a surface; coupled KE and cscK metrics}
Let us consider the complex equation \eqref{extrEqIntro} on a (compact, K\"ahler) surface $X$. This is equivalent to 
\begin{equation}\label{SurfaceEqu}
\begin{dcases}
  \Imm e^{-\ii \hat{\theta}} (\omega + \ii B)^2 = 0\\
 s(\omega) - |\gamma| \Rea e^{-\ii \hat{\theta}} \frac{(\omega + \ii B)^2}{\omega^2} = \xi_{\omega},
\end{dcases} 
\end{equation}
where $\xi_{\omega}$ satisfies
\begin{equation*}
\nabla^{1,0}_{\omega} \xi_{\omega} = V, 
\end{equation*} 
for a fixed extremal field $V$ (independent of $\omega$). By a simple computation, the topological angle $e^{\ii \hat{\theta}}$ is given by
\begin{align}\label{TopAngleSurf}
\nonumber\cos(\hat{\theta}) = \frac{\int \omega^2 - \int B^2}{\big(\big(\int \omega^2 - \int B^2 \big)^2 + 4\big(\int \omega \wedge B\big)^2\big)^{1/2}},\\
\sin(\hat{\theta}) = \frac{2\int \omega \wedge B}{\big(\big(\int \omega^2 - \int B^2 \big)^2 + 4\big(\int \omega \wedge B \big)^2\big)^{1/2}}.
\end{align}
Using the identity
\begin{equation}\label{TrigId}
\cos(\hat{\theta}) \Imm e^{-\ii \hat{\theta}} \frac{(\omega + \ii B)^2}{\omega^2} + \sin(\hat{\theta}) \Rea e^{-\ii \hat{\theta}} \frac{(\omega + \ii B)^2}{\omega^2} = \Lambda_{\omega} B, 
\end{equation}
together with the first equation in \eqref{SurfaceEqu}, we can reduce the scalar curvature equation to
\begin{equation}\label{BtwistedEqu}
s(\omega) - \frac{|\gamma|}{\sin(\hat{\theta})}\Lambda_{\omega} B= \frac{\xi_{\omega}}{\sin(\hat{\theta})}.
\end{equation} 
On the other hand, setting
\begin{equation}\label{AdjSympl}
\chi = \sin(\hat{\theta}) B + \cos(\hat{\theta}) \omega,
\end{equation}
we have an identity
\begin{equation*}
\Imm e^{-\ii \hat{\theta}} (\omega + \ii B)^2 = \frac{\chi^2 - \omega^2}{\sin(\hat{\theta})}.  
\end{equation*}
Thus, we can write \eqref{SurfaceEqu}, using the variables $\omega$, $\chi$, as
\begin{equation}\label{coupledCscK} 
\begin{dcases}
 \chi^2 = \omega^2 \\
 s(\omega) - \frac{|\gamma|}{\sin(\hat{\theta})}\Lambda_{\omega} \chi = \frac{\xi_{\omega}}{\sin(\hat{\theta})}.
 \end{dcases} 
\end{equation} 
Note that the reduction of the dHYM equation on a surface to a complex Monge-Amp\`ere was first noticed in \cite{JacobYau_special_Lag}.

Suppose that the Futaki invariant vanishes, so $\xi_{\omega}$ is a constant, and that we have $\chi > 0$. For the special choice of real coupling   
\begin{equation*}
|\gamma| = \sin(\hat{\theta}),
\end{equation*}
the equations \eqref{coupledCscK} are the \emph{coupled cscK equations} studied by Datar and Pingali \cite{DatarPingali_coupled}, in the special case of two coupled metrics $(\omega_0 = \omega, \omega_1 = \chi)$. In particular, if $X$ is del Pezzo and we have
\begin{equation*}
[\omega] + [\chi] = (1 + \cos(\hat{\theta}))[\omega] + \sin(\hat{\theta})[B] = c_1(X),
\end{equation*}
then by a standard argument the equations \eqref{coupledCscK} reduce to the \emph{coupled K\"ahler-Einstein (KE) equations} of Hultgren and Witt Nystr\"om \cite{HultgrenWitt_coupled}, namely
\begin{equation}\label{coupledKE}
\begin{dcases}
\Ric(\chi) = \Ric(\omega),\\
\Ric(\omega) = \chi + \omega.
\end{dcases}
\end{equation}
In this way the coupled KE equations for the pair $([\omega], [\chi])$ can be interpreted as fixing a canonical representative of the complexified K\"ahler class   
\begin{equation*}
\omega^{\C} = \ii [\omega] + (\sin(\hat{\theta}))^{-1}([\chi] - \cos(\hat{\theta})[\omega]),
\end{equation*}
up to the action of $\operatorname{Aut}(X, [\omega^{\C}])$.

As a different example, not related to coupled KE or cscK metrics, let us consider our general equations on a surface \eqref{SurfaceEqu} with the special ansatz
\begin{equation*}
B = \sigma \omega + \tau \Ric(\omega),\,\tau \neq 0.
\end{equation*}
Then, by \eqref{BtwistedEqu}, the scalar curvature equation is automatically satisfied, for any $\omega$, for 
\begin{equation*}
|\gamma| = \frac{\sin(\hat{\theta})}{\tau},
\end{equation*} 
so for these choices \eqref{SurfaceEqu} is equivalent to the single equation
\begin{equation}\label{DeformedCscK}
\Imm e^{-\ii \hat{\theta}} (\omega + \ii (\sigma \omega + \tau \Ric(\omega)))^2 = 0.
\end{equation}
By standard arguments, this can be solved for all sufficiently small $\tau$, provided $[\omega]$ admits a cscK representative $\omega_0$, and $(X, [\omega])$ has discrete automorphisms (this seems closely related to special cases of the deformed cscK equation studied in \cite{Dervan_crit_metrics}). Note that the corresponding large volume limit equations, in the sense of \eqref{DecoupledHYM}, are trivially solvable, by
\begin{equation*}
\begin{dcases}
\Lambda_{\omega_0} (\sigma \omega_0 + \tau\Ric(\omega_0)) = 2\sigma + \tau \hat{s}\\
s(\omega_0) = \hat{s}.
\end{dcases}
\end{equation*}

\subsection{Mean curvature and objects in $\operatorname{Fuk}(X, \omega^{\C})$}\label{MeanCurvSec}
Let $\Sigma \subset (X, \omega)$ be a Lagrangian surface. Following \cite[Section 1]{SchoenWolfson}, we define a $1$-form $\sigma_H$ on $\Sigma$ as
\begin{equation*}
\sigma_H = \iota_H \omega,
\end{equation*}
where $H$ denotes the mean curvature vector, a section of the normal bundle $N_{\Sigma|X}$. Then, by \cite[Appendix A]{SchoenWolfson} we have
\begin{equation}\label{MCform}
d\sigma_H = \Ric(\omega)|_{\Sigma}.
\end{equation}
\begin{rmk}This classical computation holds for general K\"ahler manifolds and shows that in the K\"ahler-Einstein case the mean curvature flow preserves the Lagrangian condition. 
\end{rmk}
Suppose that $(\omega, \chi)$ is a coupled KE pair on $X$, i.e., a solution of \eqref{coupledKE}, underlying a cscK metric with $B$-field $\omega^{\C} = \ii \omega + B$ (so $\chi$, given by \eqref{AdjSympl}, is a K\"ahler form by assumption). In particular, $X$ is a del Pezzo surface. Then, using \eqref{MCform}, we compute
\begin{equation*}
d\sigma_{H} = \Ric(\omega)|_{\Sigma} = (\omega + \chi)|_{\Sigma} = \sin(\hat{\theta}) B|_{\Sigma}. 
\end{equation*}
This shows that the $B$-field is exact along $\omega$-Lagrangians, and that we have a distinguished object of the Fukaya category 
\begin{equation*}
\left(\Sigma, d + \frac{\sigma_{H}}{\sin(\hat{\theta})}\right) \in \operatorname{Fuk}(X, \ii \omega + B). 
\end{equation*}
associated with the mean curvature vector of a Lagrangian. 

Similarly, if $\tilde{\Sigma} \subset (X, \chi)$ is a Lagrangian surface, with mean curvature $\tilde{H}$, we compute
\begin{equation*}
d\sigma_{\tilde{H}} = \Ric(\chi)|_{\tilde{\Sigma}} = (\omega + \chi)|_{\tilde{\Sigma}} = \left(\frac{\chi - \sin(\hat{\theta}) B}{\cos(\hat{\theta})} + \chi\right)|_{\tilde{\Sigma}} = -\tan(\hat{\theta}) B|_{\Sigma},
\end{equation*}
so again $B$ is exact along $\chi$-Lagrangians, and we have a natural object
\begin{equation*}
(\tilde{\Sigma}, d -\cot(\hat{\theta}) \sigma_{\tilde{H}}) \in \operatorname{Fuk}(X, \ii \chi + B). 
\end{equation*}

Finally, suppose that $\omega^{\C} = \ii \omega + B$, $B = \sigma \omega + \tau\Ric(\omega)$ is a cscK metric with $B$-field corresponding to a solution of \eqref{DeformedCscK}. Let $\Sigma \subset (X, \omega)$ be a Lagrangian surface. Then, we have
\begin{equation*}
d\sigma_{H} = \Ric(\omega)|_{\Sigma} = \tau^{-1} B|_{\Sigma},  
\end{equation*} 
so we find an object
\begin{equation*}
(\Sigma, d + \tau \sigma_{H}) \in \operatorname{Fuk}(X, \ii\omega + B).
\end{equation*}
  
\subsection{Continuity method}\label{C0method2d} 
Expanding the dHYM equation as
\begin{equation*}
\sin(\hat{\theta}) B^2 + 2\cos(\hat{\theta}) B\wedge \omega  = \sin(\hat{\theta})\omega^2
\end{equation*}
we note (as in \cite[Lemma 18]{SchlitzerStoppa}) that under the (semi)positivity assumptions
\begin{equation*}
\sin(\hat{\theta}), \cos(\hat{\theta}) > 0;\, B\geq 0
\end{equation*}
we obtain the elementary a priori bound
\begin{equation}\label{traceBound2d}
0 \leq \Lambda_{\omega} B < \tan(\hat{\theta}).
\end{equation}

Suppose that the classical Futaki character $\m{F}_{[\omega]} = \m{F}_{[\omega^{\C}],0}$ vanishes. With these assumptions, let us consider the continuity path
\begin{equation}\label{C0path2d}
\begin{dcases}
(\sin(\hat{\theta}) B_t + \cos(\hat{\theta}) \omega_t)^2 = \omega^2_t \\
s(\omega_t) - \frac{t|\gamma|}{\sin(\hat{\theta})}\Lambda_{\omega_t} B_t  = \hat{s} - \frac{t |\gamma|}{\sin(\hat{\theta})}\xi_{t},\,t\in [0,1],
\end{dcases}
\end{equation}
for $\omega_t \in [\omega]$, $B_t \in [B]$. The real holomorphy potential $\xi_t$ satisfies 
\begin{equation*}
\nabla^{1,0}_{\omega_t} \xi_t = V,
\end{equation*}
where the extremal field $V$ is identified with the Futaki invariant
\begin{equation*}
\m{F}_{[\omega^{\C}], 1}(W) = \int_X \varphi(W,\omega)\left(\frac{\Rea\left(\mrm{e}^{-\I\vartheta}\left(\omega+\ii B(\omega)\right)^n\right)}{\omega^n} - \hat{r}\right)\omega^n, 
\end{equation*}
with respect to the Futaki-Mabuchi inner product, which only depends on $[\omega^{\C}]$, where the constant $\hat{r}$ is such that $\m{F}$ vanishes on constant functions.

Then, for $t = 0$, the equations decouple to 
\begin{equation*}
\begin{dcases}
(\sin(\hat{\theta}) B_0 + \cos(\hat{\theta}) \omega_0)^2 = \omega^2_0 \\
s(\omega_0) = \hat{s}, 
\end{dcases}
\end{equation*}
that is, they are equivalent to the cscK equation for $\omega_0$, while for $t = 1$ they are equivalent to \eqref{SurfaceEqu}. 
 
Suppose now $(X, \omega_0)$ is a \emph{toric} cscK pair. In Section \ref{ToricEstimatesSec} (see in particular Lemma \ref{lemma:unifKstab_alpha}) we will see that, in this case, for fixed topological angle $e^{\ii \hat{\theta}}$, one can obtain a uniform a priori estimate on $\sup|\xi_t|$, for $t \in [0,1]$, depending only on the angle, of then uniform estimates of all orders on $\omega_t$, $B_t$, $\xi_t$, for $t \in [0,1]$, as long as $B_{t} \geq 0$ (so that the elementary bound \eqref{traceBound2d} holds), provided the real coupling constant satisfies
\begin{equation}\label{gammaBound}
|\gamma| < \frac{\hat{s}\,\lambda}{(1-\lambda)((\cos{\hat{\theta}})^{-1} - (\sin(\hat{\theta}))^{-1} \inf \xi_t )},
\end{equation}
where $\lambda > 0$ denotes Donaldson's uniform $K$-stability constant for the pair $(X, \omega_0)$. Therefore, the set of solution times is closed in $[0, 1]$ for $|\gamma|$ satisfying \eqref{gammaBound}. 

On the other hand, in order to obtain openness, as well as an effective estimate for the coupling constant, we will make the ansatz on the $B$-field
\begin{equation}\label{Bansatz}
B_t = \varepsilon (\omega_t + \delta \eta_t)
\end{equation}
for fixed $\varepsilon, \delta > 0$, to be determined, and suitable closed $\eta_t \in \mathcal{A}^{1,1}(X, \R)$ in the fixed cohomology class $[\eta_0]$. By \eqref{TopAngleSurf}, this is compatible with our assumptions on the topological angle $\sin(\hat{\theta}), \cos(\hat{\theta}) > 0$, provided $\varepsilon, \delta > 0$ are sufficiently small. With this ansatz, for fixed $\delta$ and to leading order in $\varepsilon$, our equations \eqref{C0path2d} reduce to
\begin{equation*}
\begin{dcases}
\Delta_{\omega_t} \eta_t = 0\\
s(\omega_t) = c,
\end{dcases}
\end{equation*}
which are solved by choosing $\omega_t = \omega_0$, $\eta_t = \eta_0$, where $\eta_0$ denotes the $\omega_0$-harmonic representative of $[\eta_0]$. Note that we have the uniform estimate for the trace term appearing in \eqref{C0path2d},
\begin{equation*}
0 \leq t\frac{\Lambda_{\omega_t} B_t}{\sin(\hat{\theta})} < \frac{1}{\cos(\hat{\theta})} < 2
\end{equation*}
for all sufficiently small $\varepsilon$, $\delta$. 

Now fix $\epsilon > 0$. Then, the general results of Section \ref{opennessSec}, combined with our uniform estimates along the path, show that the set of solution times is also open in $[0, 1]$, for all sufficiently small $\varepsilon$, $\delta$, provided the real coupling constant satisfies 
\begin{equation*} 
|\gamma| < \frac{\hat{s}\,\lambda}{2(1-\lambda)} \epsilon^{-1} < \frac{\hat{s}\,\lambda}{(1-\lambda)((\cos{\hat{\theta}})^{-1} -  (\sin(\hat{\theta}))^{-1}\inf \xi_t)},
\end{equation*}
and the semipositivity condition $B_t \geq 0$ is open. Note that, by construction, we have 
\begin{equation*}
\xi_t = \varepsilon (z + O(\delta)), 
\end{equation*}
uniformly in $t$ (recall the slope $z$ is given by \eqref{zSlope}), and so by \eqref{TopAngleSurf}, for any fixed $\epsilon > 0$ we have
\begin{equation*}
(\cos{\hat{\theta}})^{-1} -  (\sin(\hat{\theta}))^{-1}\inf \xi_t = (\cos{\hat{\theta}})^{-1} - \varepsilon (z + O(\delta)) (\sin(\hat{\theta}))^{-1} < \epsilon
\end{equation*}
for all sufficiently small $\varepsilon, \delta$. 

As for the condition $B_t \geq 0$, we note that the equation $\Delta_{\omega_t} \eta_t = 0$, together with our uniform control on $\omega_t$ (\emph{which only requires semipositivity of $B_t$}), yields uniform estimates of all orders on $\eta_t$, which are independent of sufficiently small $\varepsilon$, $\delta$, so choosing $\delta$ sufficiently small the form $B_t = \varepsilon (\omega_t + \delta \eta_t)$ \emph{actually remains strictly positive} for all $t \in [0, 1]$.    

This completes the proof of Theorem \ref{MainThmIntro} in the case of surfaces. Corollary \ref{MainCorIntro} follows from the discussion in Section \ref{KYMSec} and the fact that in the argument above we may choose $\delta$ independently of (sufficiently smal) $\varepsilon$.  

\section{Threefolds}\label{ThreefoldsSec}
In this Section we complete the proofs of Theorem \ref{MainThmIntro} and Corollary \ref{MainCorIntro} for threefolds, and we prove Theorem \ref{UnstableThmIntro}.
\subsection{Continuity method}
Suppose $(X, [\omega])$ is a compact, polarised K\"ahler threefold, which is uniformly $K$-stable with constant $\lambda$. Similarly to \eqref{TopAngleSurf}, the topological angle is determined by
\begin{equation*}
(\cos(\hat{\theta}), \sin(\hat{\theta})) = \frac{v}{||v||},
\end{equation*}
with
\begin{equation}\label{TopAngle3d}
\begin{dcases}
v_1 = \int_X  \omega^3 - 3 B^2\wedge \omega,\\
v_2 = \int_X 3 B\wedge \omega^2 - B^3.
\end{dcases}
\end{equation}
The analogue of \eqref{TrigId} is the identity
\begin{equation*} 
\cos(\hat{\theta}) \Imm e^{-\ii \hat{\theta}} \frac{(\omega + \ii B)^3}{\omega^3} + \sin(\hat{\theta}) \Rea e^{-\ii \hat{\theta}} \frac{(\omega + \ii B)^3}{\omega^3} = \Lambda_{\omega} B - \frac{B^3}{\omega^3}.  
\end{equation*}
Therefore, we can write our complex equation \eqref{extrEqIntro} on a threefold in the form
\begin{equation*}   
\begin{dcases}
  \Imm e^{-\ii \hat{\theta}} (\omega + \ii B)^3 = 0\\
 s(\omega) - \frac{|\gamma|}{\sin(\hat{\theta})} \left( \Lambda_{\omega} B - \frac{B^3}{\omega^3} \right) = \frac{\xi}{\sin(\hat{\theta})}.
\end{dcases} 
\end{equation*}
Similarly to \eqref{C0path2d}, we consider the continuity path
\begin{equation}\label{C0path3d} 
\begin{dcases}
  \Imm e^{-\ii \hat{\theta}} (\omega_t + \ii B_t)^3 = 0\\
 s(\omega_t) - \frac{t|\gamma|}{\sin(\hat{\theta})} \left( \Lambda_{\omega_t} B_t - \frac{B^3_t}{\omega^3_t} \right) = \hat{s} -\frac{t|\gamma|}{\sin(\hat{\theta})} \xi_t,\,t\in[0,1],
\end{dcases} 
\end{equation}
for $\omega_t \in [\omega]$, $B_t \in [B]$. For $t = 0$, the equations decouple to
\begin{equation*} 
\begin{dcases}
  \Imm e^{-\ii \hat{\theta}} (\omega_0 + \ii B_0)^3 = 0\\
 s(\omega_0) = \hat{s}.
\end{dcases} 
\end{equation*}
We need an analogue of the elementary a priori bound \eqref{traceBound2d}. By direct computation, we may write the dHYM equation appearing in \eqref{C0path3d} explicitly as
\begin{equation*}
-\frac{B^3}{\omega^3} \cos(\hat{\theta}) + \Lambda_{\omega}B \cos(\hat{\theta}) + 3 \frac{B^2\wedge \omega}{\omega^3}\sin(\hat{\theta}) = \sin(\hat{\theta}), 
\end{equation*}
that is, 
\begin{equation}\label{dHYM3d}
 \Lambda_{\omega}B - \frac{B^3}{\omega^3} = \left(1 - 3 \frac{B^2\wedge \omega}{\omega^3}\right)\tan(\hat{\theta}), 
\end{equation}
or, in terms of the eigenvalues $\lambda_i$ of the endomorphism $\omega^{-1} B$, as
\begin{equation*}
\sum_i \lambda_i - \prod_i \lambda_i  = \left(1 - \sum_{i < j} \lambda_i\lambda_j\right)\tan(\hat{\theta}). 
\end{equation*}
Suppose now, as in the case of surfaces, that we have
\begin{equation*}
\tan(\hat{\theta}) > 0;\,B_t\geq 0  \iff \lambda_i(t) \geq 0,\,t\in [0,1],
\end{equation*}
and moreover that we have 
\begin{equation}\label{InitialBound3d}
\sum_{i < j} \lambda_i(0)\lambda_j(0) < 1
\end{equation} 
\emph{at the initial point} $t = 0$ of the continuity path \eqref{C0path3d}. Then, we claim the uniform bound
\begin{equation}\label{BasicBound3d}
0 \leq \Lambda_{\omega_t} B_t - \frac{B^3_t}{\omega^3_t} <  \tan(\hat{\theta}),\,t\in[0,1]
\end{equation}
along the continuity path. For otherwise, by contradiction, using the dHYM equation \eqref{dHYM3d} and our assumptions above, we have
\begin{equation*}
\sum_{i < j} \lambda_i(\bar{t})\lambda_j(\bar{t}) = 1  
\end{equation*} 
and
\begin{equation*}
\sum_i \lambda_i(\bar{t}) = \prod_i \lambda_i(\bar{t})  
\end{equation*} 
for some $\bar{t} \in (0, 1]$. So necessarily $\lambda_i(\bar{t}) > 0$ for all $i$, from which
\begin{equation*}
\frac{\sum_{i < j} \lambda_i(\bar{t})\lambda_j(\bar{t})}{\prod_i \lambda_i(\bar{t})} = \frac{1}{\prod_i \lambda_i(\bar{t})} = \frac{1}{\sum_i \lambda_i(\bar{t})},  
\end{equation*}
that is,
\begin{equation*}
\left(\sum_i \lambda^{-1}_i(\bar{t})\right)^{-1} = 3\frac{\sum_i \lambda_i(\bar{t})}{3} \geq 3 \left(\frac{\sum_i \lambda^{-1}_i(\bar{t})}{3}\right)^{-1},
\end{equation*}
by the inequality between the arithmetic and harmonic means, a contradiction.

Arguing as in Section \ref{C0method2d}, using the results of Section \ref{ToricEstimatesSec}, we see that, with our current assumptions, the set of solution times for \eqref{C0path2d} is closed in $[0,1 ]$ for the range of the coupling constant 
\begin{equation*} 
|\gamma| < \frac{\hat{s}\,\lambda}{(1-\lambda)((\cos{\hat{\theta}})^{-1} -  (\sin(\hat{\theta}))^{-1}\inf \xi_t)}.
\end{equation*}
To obtain openness and effective estimates on $|\gamma|$, we consider again the ansatz \eqref{Bansatz}, namely
\begin{equation*} 
B_t = \varepsilon (\omega_t + \delta \eta_t)
\end{equation*}
for fixed $\varepsilon, \delta > 0$, to be determined, and $[\eta_t] = [\eta_0] \in H^{1,1}(X, \R)$. In this case, by construction, we have
\begin{equation*}
 \xi_t = \varepsilon(\hat{r} + O(\delta)),
\end{equation*}
uniformly in $t$, where
\begin{equation*}
\hat{r} = \frac{\int_X 3 B\wedge \omega^2 - B^3}{\int_X \omega^3},
\end{equation*}
and by \eqref{TopAngle3d}, for any fixed $\epsilon > 0$ we have
\begin{equation*}
(\cos{\hat{\theta}})^{-1} -  (\sin(\hat{\theta}))^{-1} \inf \xi_t = (\cos{\hat{\theta}})^{-1} - \varepsilon(\hat{r} + O(\delta))(\sin(\hat{\theta}))^{-1} < \epsilon
\end{equation*}
for all sufficiently small $\varepsilon, \delta$.

Note that the limit $\delta = 0$ corresponds to the trivial dHYM solution  
\begin{equation*}  
  \Imm e^{-\ii \hat{\theta}_{\epsilon}} (\omega_t + \ii \varepsilon \omega_t)^3 = 0,
\end{equation*}
valid for all $\varepsilon > 0$. The leading term in the expansion around $\varepsilon = 0$, corresponding to the large volume limit, is the Laplace equation $\Delta_{\omega_t} \eta_t = 0$.

The perturbation theory for dHYM solutions around the large volume limit has been studied in much greater generality (for all ranks) in \cite{Dervan_Zconnections}. In our particular case, the results of \cite[Section 4.1]{Dervan_Zconnections} show that, for fixed background K\"ahler form $\omega_t$ and $\delta$, the equation
\begin{equation*}
\Imm e^{-\ii \hat{\theta}_{\epsilon}} (\omega_t + \ii \varepsilon (\omega_t + \delta \eta_t))^3 = 0
\end{equation*}
is solvable with respect to $\eta_t$, with uniform estimates \emph{with respect to $\omega_t$}, for all sufficiently small $\varepsilon$. In particular, the initial bound \eqref{InitialBound3d} can be achieved. Then, the estimates on $\omega_t$ obtained in Section \ref{ToricEstimatesSec}, using the scalar curvature equation 
\begin{equation*}
s(\omega_t) - \frac{t|\gamma|}{\sin(\hat{\theta})} \left( \Lambda_{\omega_t} B_t - \frac{B^3_t}{\omega^3_t} \right) = \xi_t
\end{equation*}
and the uniform bound 
\begin{equation*}
0 \leq \frac{1}{\sin(\hat{\theta})}\left(\Lambda_{\omega_t} B_t - \frac{B^3_t}{\omega^3_t} \right) <  \frac{1}{\cos(\hat{\theta})} < 2,\,t\in[0,1],
\end{equation*}
for all sufficiently small $\varepsilon, \delta$ (which follows at once from \eqref{TopAngle3d} and \eqref{BasicBound3d}), show that the estimates on $\eta_t$ are in fact uniform with respect to $t$. These are valid for the range of the coupling constant 
\begin{equation*} 
|\gamma| < \frac{\hat{s}\,\lambda}{(1-\lambda)}\epsilon^{-1} < \frac{\hat{s}\,\lambda}{(1-\lambda)((\cos{\hat{\theta}})^{-1} - (\sin(\hat{\theta}))^{-1}\inf \xi_t)} 
\end{equation*}
for fixed $\epsilon > 0$ and all sufficiently small $\varepsilon, \delta$. The argument in Section \ref{C0method2d} then shows that the set of solution times is open and closed in $[0, 1]$.

\subsection{An unstable polarised threefold}\label{KUnstSec}
For Theorem \ref{UnstableThmIntro}, we consider the toric Fano K\"ahler-Einstein threefold
\begin{equation*}
X = \PP(\mathcal{O}\oplus\mathcal{O}(1,-1)) \to \PP^1 \times \PP^1.
\end{equation*}
The large volume limit of our equations \eqref{modelEqIntro}, i.e., the K\"ahler-Yang-Mills system \eqref{KYMIntro} on $X$, was studied in detail by Keller and T{\o}nnesen-Friedman in \cite[Section 3]{Keller_KYM}. 

They consider Calabi ansatz metrics on $X$ with respect to the product of Fubini-Study metrics on the base $\PP^1 \times \PP^1$,
\begin{equation}\label{CalabiMetric}
\omega = \frac{1}{x_1} \omega_1 + \frac{1}{x_2} \omega_2 + \ii \del\delbar f(s),\,-1 < x_i < 1.  
\end{equation}
The boundary conditions on $f(s)$ are chosen so that the cohomology class of $\omega$ is given by
\begin{equation*}
[\omega] = \frac{1}{x_1} [\omega_1] + \frac{1}{x_2} [\omega_2] + [E_0 + E_{\infty}], 
\end{equation*}
where $E_0$, $E_{\infty}$ denote the zero and infinity sections of the ruling. The Calabi ansatz provides an explicit traceless form $\alpha$, with respect to $\omega$, and one looks for K\"ahler-Yang-Mills pairs of the form $(\omega, \gamma_{a,b})$, where  \begin{equation*}
\gamma_{a, b} = a \omega + b \alpha,
\end{equation*}   
i.e. solutions of 
\begin{equation}\label{KYM_example}
s(\omega) + \hat{\gamma} \Lambda^2_{\omega}(\gamma_{a, b}\wedge \gamma_{a, b}) = \tilde{c}.
\end{equation}
Note that the cohomology class of $\gamma_{a, b}$ can be computed explicitly as 
\begin{align*}
[\gamma_{a, b}] &= \left(\frac{a}{x_1} + \frac{b(1+x_1x_2)}{(1-x^2_1)(1-x^2_2)}\right)\omega_1
 + \left(\frac{a}{x_2}+ \frac{b(1+x_1x_2)}{(1-x^2_1)(1-x^2_2)}\right) \omega_2\\
& + \left(a - \frac{b(1+x_1x_2)}{(1-x^2_1)(1-x^2_2)}\right)[E_0 + E_{\infty}].
\end{align*} 
Here the unknown is of course the metric $\omega$ and so the convex function $f(s)$, satisfying the Calabi boundary conditions. By standard theory, the scalar curvature equation for $f(s)$ can be transformed into a second order, linear ODE for the momentum profile $\phi(\tau)$, as a positive real function defined on $(-1, 1)$, with \emph{overdetermined} boundary conditions. 

Keller and T{\o}nnesen-Friedman prove, in a more general context, that the the boundary conditions can be expressed as a (nonhomogeneous) linear system in the variables
\begin{equation*}
\kappa_1 = 12 a^2 \hat{\gamma} - \tilde{c},\, \kappa_2 = 4 b^2 \hat{\gamma}, 
\end{equation*}
with coefficients and datum given by explicit rational functions of $x_1, x_2$. They observe that, for $x_1 \neq - x_2$, this linear system admits a unique solution $(\kappa_1, \kappa_2)$, given by rational functions of $x_1, x_2$, and prove that for the specific choice 
\begin{equation*} 
x_1 = \frac{1}{2}, \, x_2 = -\frac{3}{4},
\end{equation*}   
we have
\begin{equation*}
\kappa_2 > 0 \Rightarrow \hat{\gamma} > 0,\,\phi(\tau) > 0 \textrm{ on } (-1,1).
\end{equation*}
Moreover, all K\"ahler class with $x_1 = \frac{1}{2}$, $x_2 \in (-1, 1) \setminus \{-\frac{1}{2}\}$ do not admit cscK metrics. 

Following \cite[Sections 3, 4]{SchlitzerStoppa_examples}, this analysis can be generalised from \eqref{KYM_example} to the equation \eqref{modelEqIntro}. However, we will only do this here nearby the large volume limit (as in \cite[Section 7]{SchlitzerStoppa_examples}). Namely, we consider the Calabi ansatz 
\begin{equation*}
B_{\epsilon} = \varepsilon\left(\gamma_{a, b} + \ii \del\delbar g(s)\right).
\end{equation*}
The dHYM equation for $B$ becomes a second order linear ODE for the function $g(s)$, vanishing at $s = 0$, $s = \infty$. This is uniquely solvable for all sufficiently small $\varepsilon$, and the scalar curvature equation for the Calabi ansatz K\"ahler metric \eqref{CalabiMetric}, 
\begin{equation*}
s(\omega) - |\gamma| \Rea e^{-\ii \hat{\theta}} \frac{(\omega + \ii B_{\varepsilon})^n}{\omega^n} = c 
\end{equation*} 
becomes a second order linear ODE for the momentum profile $\phi_{\varepsilon}(\tau)$ as a positive real function on $(-1, 1)$, such that
\begin{equation*}
\phi_{\varepsilon}(\tau) = \phi(\tau) + \varepsilon \psi_{\epsilon}(\tau)
\end{equation*}
satisfying overdetermined boundary conditions, where $\psi_{\epsilon}(\tau)$ is continuous on $[-1, 1]$. Thus, for all sufficiently small $\varepsilon$, we find a positive solution $\psi_{\epsilon}(\tau)$, with K\"ahler parameters 
\begin{equation*}
x_1 = \frac{1}{2},\,x_2 = -\frac{3}{4} + O(\epsilon).
\end{equation*}
In particular, the corresponding K\"ahler class does not admit a cscK metric.

\section{Toric estimates}\label{ToricEstimatesSec}

In this Section, we focus on obtaining a priori estimates for solutions $\omega^{\C} = \ii \omega + B$ of \eqref{extrEqIntro}, under the assumption that $\omega$ and $B$ are invariant under the action of an~$n$-dimensional torus $\left(\bb{C}^*\right)^n\subset\mrm{Aut}(X)$. These estimates were a crucial ingredient in our proofs of Theorem \ref{MainThmIntro} and Corollary \ref{MainCorIntro}, but may be useful in a more general context.

We fix a background symplectic form $\omega_0\in [\omega]$ such that the action of the compact torus $\bb{T}^n\curvearrowright(X,\omega_0)$ is Hamiltonian. The image of the moment map is a convex polytope~$P\subset\bb{R}^n$, and any torus-invariant tensor on $X$ can be described in terms of functions on $P$. For example, the metric tensor defined by any torus-invariant $\omega\in\alpha$ corresponds to the Hessian of a convex function $u\in\m{C}^\infty(P^\circ)\cap\m{C}(\bar{P})$ with prescribed boundary behaviour. This convex function is the \emph{symplectic potential} of $\omega$: it is determined via Legendre duality from a (normalized) local K\"ahler potential for $\omega$ on the open set $U\subset X$ where the $\bb{T}^n$-action is free. We denote by $\m{S}(P)$ the set of all symplectic potentials on $P$, each of which corresponds to a torus-invariant K\"ahler metric on $X$.
 
Denoting by~$\bm{y}=(y^1,\dots,y^n)$ the usual real coordinates on~$P$, we can rewrite the real scalar curvature equation appearing in \eqref{extrEqIntro} on the momentum polytope $P$ as
\begin{equation*}
-u^{ij}_{,ij}(\bm{y})=A_0+\alpha \left(r_\omega(B)-\tilde{A}\right)(\bm{y})
\end{equation*}
where we set 
\begin{equation*}
A_0=4\hat{s},\,\alpha =4|\gamma|, 
\end{equation*}
$u\in\m{C}(P)$ is the symplectic potential of~$\omega$, the \emph{radius function} is given by
\begin{equation*}
r_\omega(B) = \left|\frac{(\omega^{\C})^n}{\omega^n}\right|,
\end{equation*} 
with (topological) average $\hat{r}$, and $\tilde{A}$ denotes a \emph{fixed affine linear function} on $P$, depending only the cohomology classes $[\omega]$, $[B]$, given by the dual of the linear function $\m{F}_{[\omega^{\C}], |\gamma|}$.

As we are assuming that the classical Futaki character of $(X, [\omega])$ vanishes, we have 
\begin{equation*}
A_0=\mrm{vol}(P,\,\mrm{d}\mu)/\mrm{vol}(\diff P,\,\mrm{d}\sigma), 
\end{equation*}
see Definition~\ref{def:Kstab} and Remark~\ref{rmk:Futaki}. 
In what follows, it will be convenient to set
\begin{equation*}
A(\omega)(\bm{y}):=A_0+\alpha \left(r_\omega(B(\omega))-\tilde{A}\right)(\bm{y})
\end{equation*}
and to write the scalar curvature equation in the form
\begin{equation}\label{eq:scalar_implicit_polytope}
-u^{ij}_{,ij}=A(\omega).
\end{equation}
Notice that~$A(\omega)\in L^\infty(P)$, as~$r_\omega(B(\omega))$ is a continuous function on~$X$.

The crucial tool for the study of prescribed curvature equations on~$P$ is uniform (toric)~$K$-stability. Its definition uses a measure~$\mrm{d}\sigma$ on~$\diff P$, given explicitly in \cite{Donaldson_stability_toric}, which is a constant multiple of the Lebesgue measure~$\mrm{d}\mu$ on each facet of~$P$.
\begin{definition}\label{def:Kstab}
For a chosen point~$p_0\in P^\circ$, we let $\m{C}_\infty(P)$ be the set of \emph{normalized} convex functions, that is convex functions~$f\in\m{C}(P)\cap\m{C}^\infty(P^\circ)$ such that $f(p_0)=\mrm{d}f(p_0)=0$. For a function~$A\in L^\infty(P)$, consider the functional
\begin{equation*}
\m{L}_A(f):=\int_{\diff P}f\,\mrm{d}\sigma-\int_P A\,f\,\mrm{d}\mu.
\end{equation*}
 The pair~$(P,A)$ is \emph{uniformly~$K$-stable} if there exists a positive constant~$\lambda$ such that
\begin{equation}\label{eq:K_uniform}
\m{L}_A(f)>\lambda\int_{\diff P}f\mrm{d}\sigma
\end{equation}
for all $f\in\m{C}_\infty(P)$, and $\m{L}_A(f)=0$ for affine-linear functions. The polytope~$P$ is called uniformly~$K$-stable if~$(P,A_0)$ is, for $A_0=\mrm{vol}(P,\,\mrm{d}\mu)/\mrm{vol}(\diff P,\,\mrm{d}\sigma)$.
\end{definition}
\begin{rmk}\label{rmk:Futaki}
The functional $\m{L}_{A_0}$ is closely related to the Futaki character of the K\"ahler class $[\omega]$: an affine linear function $f$ on the polytope corresponds to the holomorphy potential for a torus-invariant holomorphic vector field $V_f$ on $X$, and $\m{F}_{\Omega}(V_f)=\m{L}_{A_0}(f)$.
\end{rmk}
It is known that uniform~$K$-stability implies the existence of solutions to the cscK equation on the toric manifold. For the case of surfaces this was explained in a series of papers by Donaldson, culminating in \cite{Donaldson_cscKmetrics_toricsurfaces}, while the general case is a consequence of \cite{ChenCheng_existence} and \cite{He_extremal}, see also \cite{Apostolov_toric} for a detailed account of this result. Moreover, it is known that $(P,A)$-uniform stability implies a priori estimates for solutions of the prescribed scalar curvature equation (also called \emph{Abreu's equation})
\begin{equation*}
-u^{ij}_{,ij}=A.
\end{equation*}
The main technical tool to establish these estimates is~\cite[Theorem~$1.2$]{ChenCheng_estimates}, where it is shown that the norm of any solution to the prescribed scalar curvature equation on a compact K\"ahler manifold can be estimated in terms of the entropy functional~$\int\log\frac{\omega^n}{\omega_0^n}\omega^n$ and a~$\m{C}^0$-bound on the target function. In the toric setting, it is possible to obtain bounds on the entropy from uniform~$K$-stability, see~\cite[Theorem~$4.3$]{AnMinLi_prescribed_toric}. In our case the situation is slightly complicated from the fact that in~\eqref{eq:scalar_implicit_polytope} both sides of the equation depend on~$\omega$. Under some simplifying assumptions, however, the proof of~\cite[Theorem~$4.3$]{AnMinLi_prescribed_toric} gives the required result also in our case. We reproduce the proof for the reader's convenience and in order to emphasise the dependence on $\sup\abs{A(\omega)}$.
\begin{prop}[Li-Lian-Sheng \cite{AnMinLi_prescribed_toric}]\label{prop:entropy_estimate}
Assume that~$\omega\in[\omega_0]$ is a torus-invariant solution of~\eqref{eq:scalar_implicit_polytope}, and that~$(P,A(\omega))$ is~$\lambda$-uniformly~$K$-stable, with~$\lambda$ possibly depending on~$\omega$. Then the entropy
\begin{equation*}
\int_M\log\frac{\omega^n}{\omega_0^n}\omega^n
\end{equation*}
can be estimated in terms of~$\lambda$,~$\sup\abs{A(\omega)}$, and~$\omega_0$.
\end{prop}
\begin{proof} 
Recall that the~$K$-energy functional can be decomposed as
\begin{equation*}
\m{M}(\omega)=\int_M\log\left(\frac{\omega^n}{\omega_0^n}\right)\frac{\omega^n}{n!}+\m{J}_{-\mrm{Ric}(\omega_0)}(\omega).
\end{equation*}
On a toric manifold, it was shown in~\cite{Donaldson_stability_toric} that the~$K$-energy on torus-invariant metrics can be written in terms of the corresponding symplectic potential on~$P$ as follows: define for every~$A\in L^\infty(P)$,~$f\in\m{C}^\infty(P)$ and symplectic potential~$v$
\begin{equation*}
\m{F}_{A_0}(v)=-\int_P\log\det(D^2v)\,\mrm{d}\mu+\m{L}_{A_0}(u).
\end{equation*}
Then, the~$K$-energy is
\begin{equation*}
\m{M}(\omega)=(2\pi)^n\m{F}_{A_0}(u)
\end{equation*}
where~$u$ is the symplectic potential on~$P$ corresponding to~$\omega$. These identities show that we can bound the entropy if we have estimates for~$\m{J}_{-\mrm{Ric}(\omega_0)}(\omega)$ and~$\m{F}_{A_0}(u)$. To do so, we will first give an estimate for~$\m{F}_{A(\omega)}(u)$. We will then relate~$\m{F}_{A_0}(u)$ to~$\m{F}_{A(\omega)}(u)$. The inequalities involved will also allow us to give an estimate for~$\m{J}_{-\mrm{Ric}(\omega_0)}(\omega)$.

Let~$u_0$ be the (normalized) symplectic potential corresponding to~$\omega_0$. As~$u$ solves equation~\eqref{eq:scalar_implicit_polytope},~$u$ is a minimizer of~$\m{F}_{A(\omega)}$ on the set on normalized potentials by~\cite[Proposition~$3.3.4$]{Donaldson_stability_toric}, and uniform~$K$-stability of~$(P,A(\omega))$ gives
\begin{equation*}
\m{F}_{A(\omega)}(u)\leq\m{F}_{A(\omega)}(u_0)\leq C(\omega_0,\lambda).
\end{equation*}
The proof of~\cite[Proposition~$5.1.2$]{Donaldson_stability_toric} shows that~$\m{F}_{A(\omega)}$ can also be bounded below in terms of~$\lambda$,~$\sup\abs{A(\omega)-A_0}$ and~$\omega_0$.

From the bound~$\abs*{\m{F}_{A(\omega)}(u)}<C(\omega_0,\lambda,\sup\abs{A(\omega)})$ we can obtain a similar estimate for~$\m{F}_{A_0}(u)$, but we first need some preliminary inequalities. By~\cite[Lemma~$5.1.3$]{Donaldson_stability_toric} there is a constant~$C$ such that, for any normalized symplectic potential~$v$,
\begin{equation}\label{eq:stima_interno}
\int_Pv\,\mrm{d}\mu\leq C\int_{\diff P}v\,\mrm{d}\sigma.
\end{equation}
As~$u$ satisfies Abreu's equation~\eqref{eq:scalar_implicit_polytope},~\cite[Corollary~$2$]{Donaldson_Abreu_IntEst} shows
\begin{equation}\label{eq:stima_bordo}
\int_{\diff P}u\,\mrm{d}\sigma\leq\frac{n}{\lambda}.
\end{equation}
Then we can estimate~$\m{F}_{A_0}(u)$ combining~\eqref{eq:stima_interno} and~\eqref{eq:stima_bordo}
\begin{equation*}
\begin{split}
\abs*{\m{F}_{A_0}(u)}=&\abs*{\m{F}_{A_0}(u)-\m{F}_{A(\omega)}(u)+\m{F}_{A(\omega)}(u)}\leq\abs*{\m{L}_{A_0}(u)-\m{L}_{A(\omega)}(u)}+\abs*{\m{F}_{A(\omega)}(u)}\\
\leq&\sup\abs*{A(\omega)-A_0}\int_Pu\,\mrm{d}\mu+\abs*{\m{F}_{A(\omega)}(u)}\leq C(\omega_0,\lambda,\sup\abs{A(\omega)}).
\end{split}
\end{equation*}

It remains to show that~$\m{J}_{-\mrm{Ric}(\omega_0)}(\omega)$ can also be bounded by~$\omega_0$,~$\lambda$, and~$\sup\abs{A(\omega)}$. The~$\m{J}$-functional can be estimated in terms of the~$d_1$-distance on K\"ahler potentials, see for example~\cite[Lemma~$4.4$]{ChenCheng_existence}. Letting~$\omega=\omega_0+\I\diff\bdiff\varphi$ for a potential~$\varphi$ then we can write
\begin{equation*}
\abs*{\m{J}_{-\mrm{Ric}(\omega_0)}(\omega)}\leq C(\omega_0)\,d_1(0,\varphi)
\end{equation*}
for a constant depending only on the background metric. On the other hand, the~$d_1$-distance has a particularly simple expression in the toric setting, as geodesics in the space of torus-invariant K\"ahler potentials on $X$ correspond to line segments in $\m{S}(P)$ under the Legendre transform. Recall that symplectic potentials corresponding to~$\omega_0$ and~$\omega$ are~$u_0$ and~$u$, respectively. Then we have
\begin{equation*}
d_1(0,\varphi)=\left(\frac{\pi}{2}\right)^n\int_P\abs{u_0-u}\mrm{d}\mu
\end{equation*}
so we can estimate~$\m{J}$ by
\begin{equation*}
\abs*{\m{J}_{-\mrm{Ric}(\omega_0)}(\omega)}\leq C(\omega_0)\int_P\abs{u_0-u}\mrm{d}\mu\leq C(\omega_0,\lambda)
\end{equation*}
using~\eqref{eq:stima_interno} and~\eqref{eq:stima_bordo}. Together with the estimate on~$\m{F}_{A_0}(u)$, this shows that the entropy functional is bounded in terms of~$\sup\abs{A(\omega)}$,~$\lambda$, and the background metric.
\end{proof}
\begin{lemma}\label{lemma:unifKstab_alpha}
Assume that~$P$ is~$\lambda$-uniformly~$K$-stable, and that there exists a constant $R>0$ such that, for every~$\omega\in[\omega_0]$ and a solution~$B(\omega)$ of the dHYM equation, we have
\begin{equation}\label{eq:radius_bounds}
r_\omega(B(\omega)) <R.
\end{equation}
If $\alpha$ satisfies
\begin{equation*}
\alpha<\frac{A_0\,\lambda}{4(1-\lambda)(R-\inf \tilde{A} )}
\end{equation*}
then~$(P,A(\omega))$ is~$\lambda'$-uniformly~$K$-stable, where~$\lambda'$ depends only on~$\lambda$,~$\alpha$ and~$R$.
\end{lemma}
\begin{rmk}
If we allowed~$\alpha<0$, the same conclusion would hold \emph{without} assuming the bound \eqref{eq:radius_bounds} on the radius.
\end{rmk}
\begin{proof}[Proof of Lemma \ref{lemma:unifKstab_alpha}]
Let~$f$ be a normalized convex function on the polytope. As~$P$ is~$\lambda$-uniformly~$K$-stable, we have
\begin{equation*}
\begin{split}
\int_{\diff P}f\,\mrm{d}\sigma-\int_PA(\omega)f\,\mrm{d}\mu=&\int_{\diff P}f\,\mrm{d}\sigma-\int_PA_0f\,\mrm{d}\mu-4\,\alpha\int_P\left(r_\omega(B(\omega))-\tilde{A}\right)f\,\mrm{d}\mu>\\
>&\lambda\int_{\diff P}f\,\mrm{d}\sigma-4\,\alpha\int_P\left(r_\omega(B(\omega))-\tilde{A}\right)f\,\mrm{d}\mu.
\end{split}
\end{equation*}
Notice that~$\lambda<1$, as~$A_0>0$ and any normalized function is nonnegative. Using \eqref{eq:radius_bounds}, $\alpha>0$, this implies
\begin{equation*}
\begin{split}
\int_{\diff P}f\,\mrm{d}\sigma-\int_PA(\omega)f\,\mrm{d}\mu>&\lambda\int_{\diff P}f\,\mrm{d}\sigma-4\frac{\alpha(R-\inf \tilde{A})}{A_0}\int_PA_0f\,\mrm{d}\mu.
\end{split}
\end{equation*}
We can use again the~$\lambda$-uniform~$K$-stability of~$P$ to finally obtain
\begin{equation*}
\begin{split}
\int_{\diff P}f\,\mrm{d}\sigma-\int_PA(\omega)f\,\mrm{d}\mu>&\lambda\int_{\diff P}f\,\mrm{d}\sigma+4\frac{\alpha(R-\inf \tilde{A} )}{A_0}(\lambda-1)\int_{\diff P}f\,\mrm{d}\sigma
\end{split}
\end{equation*}
so that to deduce~$(P,A(\omega))$-uniform~$K$-stability it will be sufficient that
\begin{equation*}
\lambda'=\lambda+4\frac{\alpha(R-\inf \tilde{A})}{A_0}(\lambda-1)>0.\qedhere
\end{equation*}
\end{proof}
\section{Openness}\label{opennessSec}
This Section proves the openness result for the cscK equation with $B$-field \eqref{modelEqIntro}, which is used in the proofs of Theorem \ref{MainThmIntro} and Corollary \ref{MainCorIntro}.

We need to consider the linearisation of \eqref{extrEqIntro}, in the toric case, assuming that the classical Futaki character $\m{F}_{\omega}$ vanishes. Let us denote by $\Omega$ the set of K\"ahler forms in the fixed cohomology class $[\omega]$. It will be useful to introduce the operator $Q:\Omega\to\m{A}^{2n}(X)$ given by
\begin{equation*}
Q(\omega)=\left(s(\omega)-\xi\right)\omega^n-|\gamma|\,\Rea\left(\mrm{e}^{-\I\hat{\vartheta}}\left(\omega+\I B(\omega)\right)^n\right)
\end{equation*}
where as usual $B(\omega)$ is the unique solution of the dHYM equation, which is well defined in our case, so that the real part of \eqref{extrEqIntro} is equivalent to $Q(\omega)=0$. The constant $\hat{\vartheta}$ and function $\xi$ are such that $Q(\omega)$ integrates to zero on $X$; we denote by $\m{A}^{2n}_0(X)$ the set of such top-degree forms. Note that $\hat{\vartheta}$ and $c$ depend only on the classes $[B]$ and $[\omega]$, and on the coupling constant $|\gamma|$.

For $\varepsilon>0$, we consider the rescaled system of equations, to be solved for $B$ and $\omega$ in fixed cohomology classes,
\begin{equation}\label{eq:dHYMcscK_R}
\begin{dcases}
\frac{1}{\varepsilon}\Imm\left(\mrm{e}^{-\I\hat{\vartheta}_\varepsilon}\left(\omega+\I\,\varepsilon B\right)^n\right)=0\\
\left(s(\omega)-\xi_\varepsilon\right)\omega^n-|\gamma|\,\Rea\left(\mrm{e}^{-\I\hat{\vartheta}_\varepsilon}\left(\omega+\I\,\varepsilon B\right)^n\right)=0.
\end{dcases}
\end{equation}
For our applications, we are interested in the behaviour of \eqref{eq:dHYMcscK_R} for $\varepsilon\ll 1$. By \cite[Section $2.7$]{SchlitzerStoppa}, we can write \eqref{eq:dHYMcscK_R} in the form
\begin{equation}\label{eq:dHYMcscK_infty}
\begin{dcases}
\left(\Lambda_\omega B- z \right)\omega^n+O(\varepsilon)=0\\
\left(s(\omega)-\hat{s}\right)\omega^n+O(\varepsilon)=0.
\end{dcases}
\end{equation}
If we let $B_\varepsilon(\omega)$ be the solution to the rescaled dHYM equation $\Imm\left(\mrm{e}^{-\I\hat{\vartheta}_\varepsilon}\left(\omega+\I\,\varepsilon B\right)^n\right)=0$ then $B_\varepsilon(\omega)$ converges, for $\varepsilon\to 0$, to the solution of $\Lambda_\omega B =z$. So, if we consider the operator
\begin{equation*}
Q_{\varepsilon}(\omega)=\left(s(\omega)-\xi_{\varepsilon}\right)\omega^n-|\gamma| \Rea\left(\mrm{e}^{-\I\hat{\vartheta}}\left(\omega+\I\,\varepsilon B_\varepsilon(\omega)\right)^n\right)
\end{equation*}
the asymptotic expansion gives $Q_\varepsilon(\omega)-Q_0(\omega)=O(\varepsilon)$, for
\begin{equation*}
Q_0(\omega)=\left(s(\omega)-\hat{s}\right)\omega^n.
\end{equation*}
Let us specialise to the toric case. By a slight abuse of notation, we still denote by $Q_\varepsilon$, $Q_0$ the functionals defined on $\m{S}(P)$ by $u\mapsto Q_\varepsilon(\omega(u))$. Consider in particular the differential of $Q_0$,
\begin{equation*}
\left(DQ_0\right)_{u}:T_u\m{S}(P)=\m{C}^\infty(P)\to\m{A}^{2n}_0(X).
\end{equation*}
Any $u\in\m{S}(P)$ defines a K\"ahler form $\omega(u)$, we can use $u$ to normalize functions on $X$, thus identifying $\m{A}^{2n}_0(X)$ with $\m{C}^\infty_0(X)$. Notice that this normalization is equivalent to identifying $\m{A}^{2n}_0(X)$ with the space $\m{C}^\infty_0(P)$ of smooth functions on $P$ with vanishing integral with respect to the Lebesgue measure.

We recall a simple feature of the linearization of Abreu's equation, a consequence of the integration by parts formula in \cite[Lemma $3.3.5$]{Donaldson_stability_toric}.
\begin{lemma}\label{lem:Qinfty_invert}
Assume that the Futaki invariant of $P$ vanishes. Then the linearization of $Q_0$ around any $u\in\m{S}(P)$ is surjective on the space $\mrm{Aff}(P)^\perp$ of smooth functions on $P$ that are $L^2$-orthogonal to affine-linear functions.
\end{lemma}
As being surjective is an open property for bounded operators on Banach spaces, we can deduce openness along the continuity paths \eqref{C0path2d}, \eqref{C0path3d}.
\begin{prop}
Consider the family of operators $Q_{\varepsilon,t}:\m{S}(P)\to\m{C}^\infty(P)$ corresponding to the continuity paths \eqref{C0method2d}, \eqref{C0path3d}. Then, for $\varepsilon\ll 1$, the set of $t\in[0,1]$ for which there exists solutions of the equation $Q_{\varepsilon,t}(u)=0$ is open.
\end{prop}
\begin{proof} Fix $\bar{t}\in[0,1]$ for which $Q_{\varepsilon,\bar{t}}(u)=0$ is solvable. We claim that for sufficiently small $\varepsilon$, uniformly in $\bar{t}$, the map
\begin{equation*}
\left(DQ_{\varepsilon, \bar{t}}\right)_{\omega_{\bar{t}}}\!:\m{C}^\infty(P)\to\mrm{Aff}(P)^\perp
\end{equation*}
is surjective. The following version of the Implicit Function Theorem (IFT) will then guarantee that, for $t$ close enough to $\bar{t}$, the equation $Q_{\varepsilon,t}=0$ is solvable. 
\begin{thm}[Implicit Function Theorem]
Let $F:X\times Y\to Z$ be a smooth map between Banach spaces. Let $(x,y)\in X\times Y$ be a point such that $F(x,y)=0$ and $D_YF_{x,y}:TY\to TZ$ is surjective. Then there are neighbourhoods $U$ of $x$ and $V$ of $y$ such that for every $x'\in U$ there exists $y'\in V$ such that $F(x',y')=0$.
\end{thm}
This can be obtained from the usual Inverse Function Theorem for Banach space surjections applied to the operator $\hat{F}(x,y):=\left(x,F(x,y)\right)$. To apply the IFT we should first extend the domain of $Q_{\varepsilon,t}$ to a Banach space; we can assume that our operators are defined for $2$-forms $B$ and symplectic potentials $u$ with $\m{C}^{k,\alpha}$ regularity, for some $k\geq 2$, $0<\alpha<1$. The IFT will just guarantee the existence of solutions with $\m{C}^{k,\alpha}$ regularity, but the regularity theory for elliptic equations tells us that a $\m{C}^{k,\alpha}$-solution will actually be smooth.

As for the claim, recall we have
\begin{equation*}
Q_{\varepsilon,t}=-u^{ij}_{,ij} - \hat{s} + O(\varepsilon),\,DQ_{\varepsilon,t}=DQ_{0,t}+O(\varepsilon).
\end{equation*}
Lemma \ref{lem:Qinfty_invert} tells us that the differential of $Q_{0,t}$ at the point $u_0\in\m{S}(P)$ is surjective on $\mrm{Aff}(P)^\perp$. In particular, $DQ_{0,t}$ has a right inverse, and so $ DQ_{\varepsilon,t}$ will also have a right inverse, if $\norm*{  DQ_{\varepsilon,t}-DQ_{0,t}}$ is small enough, depending on the norm of a right inverse for $DQ_{0,t}$. As $  DQ_{\varepsilon,t}-DQ_{0,t}=O(\varepsilon)$, the operator $DQ_{\varepsilon,t}$ will have a right inverse for small enough $\varepsilon$. Finally, by our estimates in Sections \ref{SurfacesSec}, \ref{ThreefoldsSec}, $\varepsilon$ can be chosen uniformly with respect to $\bar{t}$. 
\end{proof}

\section{A moment map for general~$B$-field classes}\label{momentMapSec}

In this Section we explain and prove our moment map interpretation for~\eqref{modelEqIntro} in the case of non-Hodge~$B$-field classes, Theorem~\ref{momentMapThm}. While the general setup is close to~\cite{AlvarezGarciaGarcia_KYM, CollinsYau_dHYM_momentmap, SchlitzerStoppa}, the details in this general case are rather technical.

We start by recalling the interpretation of the dHYM equation as a moment map form~\cite{CollinsYau_dHYM_momentmap}. We slightly generalize the construction, in order to avoid assuming that~$[\beta]$ is a Hodge class.

Consider the space~$\m{A}:=\m{A}^1(X,\bb{R})$, and for a fixed~$2$-form~$\beta$ define a map~$F:\m{A}\to[\beta]$ by~$ F(y)=\beta+\mrm{d}y$. We can define a (closed)~$2$-form on~$\m{A}$ by
\begin{equation*}
\Omega_y(\alpha_1,\alpha_2)=-n\int_X\alpha_1\wedge\alpha_2\wedge\Rea\left(\mrm{e}^{-\ii\hat{\vartheta}}\left(\omega+\ii F(y)\right)^{n-1}\right).
\end{equation*}
At least formally, in a neighbourhood of a solution~$y_0\in\m{A}$ of dHYM the~$2$-form~$\Omega$ is nondegenerate (see~\cite{CollinsYau_dHYM_momentmap, Dervan_Zconnections}).
\begin{lemma}
The (commutative) Lie algebra~$\f{g}:=\m{C}^\infty(X,\bb{R})$ acts on~$\m{A}$ by
\begin{equation*}
\hat{\varphi}_y:=\mathrm{d}\varphi\in T_y\m{A}
\end{equation*}
and the action is Hamiltonian, with equivariant moment map
\begin{equation*}
\mu(y)=\Imm\left(\mrm{e}^{-\ii\hat{\vartheta}}\left(\omega+\ii F(y)\right)^n\right)
\end{equation*}
where we are using the natural pairing between~$\m{C}^\infty(X,\bb{R})$ and~$2n$-forms to identify the moment map with a top-degree form. 
\end{lemma}
The proof is essentially the same as in~\cite[Section~$2.1$]{CollinsYau_dHYM_momentmap}.

\subsection{Extending the moment map}

Let now~$\f{h}\subset\Gamma(X,TX)$ be the Lie algebra of Hamiltonian vector fields on~$X$, with respect to the background symplectic form~$\omega$. We consider the Lie algebra extension
\begin{equation}\label{eq:g_ext}
\begin{tikzcd}
0 \arrow[r] & \f{g} \arrow[r, "\iota"] & \tilde{\f{g}}:=\f{g}\times\f{h} \arrow[r, "p"] & \f{h} \arrow[r] & 0
\end{tikzcd}
\end{equation}
where on~$\tilde{\f{g}}$ we consider the bracket
\begin{equation}\label{eq:bracket}
\Big[(f_1,V_1),(f_2,V_2)\Big]:=\Big(V_1(f_2)-V_2(f_1)-\beta(V_1,V_2),[V_1,V_2]\Big).
\end{equation}
\begin{lemma}\label{lem:bracket}
Equation~\eqref{eq:bracket} defines a Lie bracket on~$\tilde{\f{g}}$.
\end{lemma}
\begin{proof}
The only nontrivial property to check is the Jacobi identity. Let~$\zeta_i=(f_i,V_i)\in\tilde{\f{g}}$, for~$i=1,2,3$. As the bracket in the second component of~$\tilde{\f{g}}$ is just the bracket on vector fields, it will be enough to prove that
\begin{equation*}
\pi_1\left(\big[\zeta_1,[\zeta_2,\zeta_3]\big]+\big[\zeta_2,[\zeta_3,\zeta_1]\big]+\big[\zeta_3,[\zeta_1,\zeta_2]\big]\right)=0
\end{equation*}
where~$\pi_1:\tilde{\f{g}}\to\f{g}$ is the projection on the first component. The first term is
\begin{equation*}
\pi_1\left(\big[\zeta_1,[\zeta_2,\zeta_3]\big]\right)=
V_1\left(V_2(f_3)\right)-V_1\left(V_3(f_2)\right)-V_1\left(\beta(V_2,V_3)\right)-[V_2,V_3](f_1)-\beta(V_1,[V_2,V_3]).
\end{equation*}
Performing cyclic permutations and using the six-terms formula for the derivative of a~$2$-form, one finds
\begin{equation*}
\pi_1\left(\big[\zeta_1,[\zeta_2,\zeta_3]\big]+\big[\zeta_2,[\zeta_3,\zeta_1]\big]+\big[\zeta_3,[\zeta_1,\zeta_2]\big]\right)=-\mrm{d}\beta(V_1,V_2,V_3)=0
\end{equation*}
as~$\beta$ is closed by assumption.
\end{proof}
For each~$y\in\m{A}$ we define a splitting of~\eqref{eq:g_ext} (as a sequence of vector spaces)~$\vartheta_y:\tilde{\f{g}}\to\f{g}$, by setting~$\vartheta_y(f,V):=f+y(V)$. We define an infinitesimal action of~$\tilde{\f{g}}$ on~$\m{A}$ as
\begin{equation}\label{eq:inf_action_ext}
\hat{\zeta}_y:=\widehat{\vartheta_y(\zeta)}_y+p(\zeta)\lrcorner F(y).
\end{equation}
\begin{lemma}\label{lem:Lie_action}
The correspondence~\eqref{eq:inf_action_ext} defines a Lie algebra action on~$\m{A}$.
\end{lemma}
\begin{proof}
We need to check that for~$\zeta_1,\zeta_2\in\tilde{\f{g}}$ and~$y\in\m{A}$ we have~$\widehat{[\zeta_1,\zeta_2]}=-[\hat{\zeta_1},\hat{\zeta_2}]$. Letting~$\zeta_1=(f_1,V_1)$ and~$\zeta_2=(f_2,V_2)$, we find
\begin{equation}\label{eq:action_commutator}
\widehat{[\zeta_1,\zeta_2]}_y=\mrm{d}\big(V_1(f_2)-V_2(f_1)-\beta(V_1,V_2)\big)+[V_1,V_2]\lrcorner\beta+\m{L}_{[V_1,V_2]}y.
\end{equation}
On the other hand, for~$i=1,2$
\begin{equation*}
\hat{\zeta_i}_y=\mrm{d}f_i+\mrm{d}(y(V_i))+V_i\lrcorner(\beta+\mrm{d}y)=\mrm{d}f_i+V_i\lrcorner\beta+\m{L}_{V_i}y.
\end{equation*}
So we can compute the commutator as
\begin{equation}\label{eq:comm_part1}
\left[\hat{\zeta_1},\hat{\zeta_2}\right]_y=\mrm{d}\left(V_2(f_1)-V_1(f_2)\right)+\m{L}_{[V_2,V_1]}y+\m{L}_{V_2}\left(V_1\lrcorner\beta\right)-\m{L}_{V_1}\left(V_2\lrcorner\beta\right).
\end{equation}
Using Cartan's formula, as the Lie derivative commutes with contraction we obtain
\begin{equation}\label{eq:lie_contraction}
\begin{split}
\m{L}_{V_2}\left(V_1\lrcorner\beta\right)-\m{L}_{V_1}\left(V_2\lrcorner\beta\right)=&V_2\lrcorner\mrm{d}\left(V_1\lrcorner\beta\right)+\mrm{d}\left(V_2\lrcorner V_1\lrcorner\beta\right)-\m{L}_{V_1}\left(V_2\lrcorner\beta\right)=\\
=&\mrm{d}\left(\beta(V_1,V_2)\right)-[V_1,V_2]\lrcorner\beta.
\end{split}
\end{equation}
Putting~\eqref{eq:comm_part1} and~\eqref{eq:lie_contraction} together, we obtain precisely the opposite of~\eqref{eq:action_commutator}.
\end{proof}

The general results of~\cite{AlvarezGarciaGarcia_KYM} give a sufficient condition for the action defined by~\eqref{eq:inf_action_ext} to be Hamiltonian, with respect to the symplectic form~$\Omega$. To state the condition, for any~$y\in\m{A}$ define~$\vartheta_y^\perp:\f{h}\to\tilde{\f{g}}$ by the equation
\begin{equation*}
\bm{1}_{\tilde{\f{g}}}=\iota\circ\vartheta_y+\vartheta_y^\perp\circ p.
\end{equation*}
\begin{prop}[cf.\cite{AlvarezGarciaGarcia_KYM}, Proposition 1.3]\label{thm:GarciaFernandez}
The action of~$\tilde{\f{g}}$ on~$\m{A}$ is Hamiltonian if and only if there is a map~$\sigma:\m{A}\to\f{h}^*$ such that for all~$V\in\f{h}$
\begin{equation}\label{eq:cond_sigma}
\widehat{\,\vartheta^\perp V\,}\!\lrcorner\Omega=\langle\mu,\left(\mrm{d}\vartheta\right)(V)\rangle+\mrm{d}\langle\sigma,V\rangle.
\end{equation}
In this case, a moment map is given by~$\langle\tilde{\mu},\Psi\rangle=\langle\mu,\vartheta(\Psi)\rangle+\langle\sigma,p(\Psi)\rangle.$
\end{prop}
We can use this result to produce a moment map for the action of~$\tilde{\f{g}}$.
\begin{prop}
In our situation, the map~$\sigma$ defined by
\begin{equation*}
\langle\sigma(y),X^\omega_\varphi\rangle:=\int_X\varphi\,\Rea\left(\mrm{e}^{-\ii\hat{\vartheta}}\left(\omega+\ii F(y)\right)^{n}\right)
\end{equation*}
satisfies~\eqref{eq:cond_sigma}. The corresponding moment map is~$\tilde{\f{g}}$-equivariant, i.e. for any~$\zeta_1,\zeta_2\in\tilde{\f{g}}$
\begin{equation}\label{equivariance}
\left\langle\tilde{\mu},[\zeta_1,\zeta_2]\right\rangle=\Omega(\hat{\zeta}_1,\hat{\zeta}_2).
\end{equation}
\end{prop}
Checking that $\sigma$ satisfies \eqref{eq:cond_sigma} is completely analogous to the proof of Theorem $2$ in~\cite{SchlitzerStoppa}. It is necessary instead to check the equivariance claim.
\begin{proof} The expression on the left hand side of~\eqref{equivariance} is 
\begin{equation*}
\left\langle\tilde{\mu},[\zeta_1,\zeta_2]\right\rangle=\langle\mu,\vartheta([\zeta_1,\zeta_2])\rangle+\langle\sigma,p([\zeta_1,\zeta_2])\rangle.
\end{equation*}
Writing~$\zeta_i=(f_i,V_i)$, for~$y\in\m{A}$ we have
\begin{align*}
&\vartheta([\zeta_1,\zeta_2])=V_1(f_2)-V_2(f_1)-\beta(V_1,V_2)+y([V_1,V_2]),\\
& p([\zeta_1,\zeta_2])=[V_1,V_2].
\end{align*}
Recall that the Hamiltonian potential of~$[V_1,V_2]$ is~$\omega(V_2,V_1)$. Then,
\begin{align}\label{eq:equiv_mu}
& \nonumber\left\langle\mu(y),\vartheta([\zeta_1,\zeta_2])\right\rangle=\int\left(V_1(f_2)-V_2(f_1)-\beta(V_1,V_2)+y([V_1,V_2])\right)\Imm\left(\mrm{e}^{-\ii\hat{\vartheta}}(\omega+\ii\,F(y))^n\right),\\
& \langle\sigma,p([\zeta_1,\zeta_2])\rangle=\int \omega(V_2,V_1)\Rea\left(\mrm{e}^{-\ii\hat{\vartheta}}(\omega+\ii\,F(y))^n\right).
\end{align}
For~$\Omega(\hat{\zeta}_1,\hat{\zeta}_2)$ instead we have, since~$\Rea\left(\mrm{e}^{-\ii\hat{\vartheta}}(\omega+\ii\,F(y))^{n-1}\right)$ is closed,
\begin{equation}\label{eq:equiv_Omega}
\begin{split}
\Omega(\hat{\zeta}_1,\hat{\zeta}_2)=&-n\int\mrm{d}\left(f_1+y(V_1)\right)\wedge\left(V_2\lrcorner F(y)\right)\wedge\Rea\left(\mrm{e}^{-\ii\hat{\vartheta}}(\omega+\ii\,F(y))^{n-1}\right)\\
&-n\int\left(V_1\lrcorner F(y)\right)\wedge\mrm{d}\left(f_2+y(V_2)\right)\wedge\Rea\left(\mrm{e}^{-\ii\hat{\vartheta}}(\omega+\ii\,F(y))^{n-1}\right)\\
&-n\int\left(V_1\lrcorner F(y)\right)\wedge\left(V_2\lrcorner F(y)\right)
\wedge\Rea\left(\mrm{e}^{-\ii\hat{\vartheta}}(\omega+\ii\,F(y))^{n-1}\right).
\end{split}
\end{equation}
Now,~$\mrm{d}\left(f_1+y(V_1)\right)\wedge\Imm\left(\mrm{e}^{-\ii\hat{\vartheta}}(\omega+\ii\,F(y))^{n}\right)=0.$ Contracting with~$V_2$ and integrating we obtain
\begin{equation*}
\begin{split}
-\int V_2\left(f_1+y(V_1)\right)&\,\Imm\left(\mrm{e}^{-\ii\hat{\vartheta}}(\omega+\ii\,F(y))^{n}\right)=\\
&=-n\int\mrm{d}\left(f_1+y(V_1)\right)\wedge(V_2\lrcorner F(y))\wedge\Rea\left(\mrm{e}^{-\ii\hat{\vartheta}}(\omega+\ii\,F(y))^{n-1}\right).
\end{split}
\end{equation*}
So that~\eqref{eq:equiv_Omega} can be rewritten as
\begin{equation*}
\begin{split}
\Omega(\hat{\zeta}_1,\hat{\zeta}_2)=&
\int\big(V_1(f_2)-V_2(f_1)+V_1(y(V_2))-V_2(y(V_1))\big)\Imm\left(\mrm{e}^{-\ii\hat{\vartheta}}(\omega+\ii\,F(y))^{n}\right)\\
&-n\int\left(V_1\lrcorner F(y)\right)\wedge\left(V_2\lrcorner F(y)\right)
\wedge\Re\left(\mrm{e}^{-\ii\hat{\vartheta}}(\omega+\ii\,F(y))^{n-1}\right).
\end{split}
\end{equation*}
Compare this with~\eqref{eq:equiv_mu} to obtain
\begin{equation*}
\begin{split}
\left\langle\mu(y),\vartheta([\zeta_1,\zeta_2])\right\rangle-\Omega(\hat{\zeta}_1,\hat{\zeta}_2)=&-\int F(y)(V_1,V_2)\,\Imm\left(\mrm{e}^{-\ii\hat{\vartheta}}(\omega+\ii\,F(y))^n\right)+\\
+n\int(&V_1\lrcorner F(y))\wedge\left(V_2\lrcorner F(y)\right)
\wedge\Rea\left(\mrm{e}^{-\ii\hat{\vartheta}}(\omega+\ii\,F(y))^{n-1}\right).
\end{split}
\end{equation*}
Contracting with~$V_2$ the form~$(V_1\lrcorner F(y))\wedge\Imm\left(\mrm{e}^{-\ii\hat{\vartheta}}(\omega+\ii\,F(y))^n\right)=0$, we find
\begin{equation*}
\begin{gathered}
F(y)(V_1,V_2)\Imm\left(\mrm{e}^{-\ii\hat{\vartheta}}(\omega+\ii\,F(y))^n\right)-n(V_1\lrcorner F(y))\wedge(V_2\lrcorner\omega)\wedge\Imm\left(\mrm{e}^{-\ii\hat{\vartheta}}(\omega+\ii\,F(y))^{n-1}\right)\\
-n(V_1\lrcorner F(y))\wedge(V_2\lrcorner F(y))\wedge\Rea\left(\mrm{e}^{-\ii\hat{\vartheta}}(\omega+\ii\,F(y))^{n-1}\right)=0.
\end{gathered}
\end{equation*}
Similarly, we have
\begin{equation*}
\begin{gathered}
\omega(V_2,V_1)\,\Rea\left(\mrm{e}^{-\ii\hat{\vartheta}}(\omega+\ii\,F(y))^n\right)-n(V_2\lrcorner\omega)\wedge(V_1\lrcorner\omega)\wedge\Re\left(\mrm{e}^{-\ii\hat{\vartheta}}(\omega+\ii\,F(y))^{n-1}\right)\\
+n(V_2\lrcorner\omega)\wedge(V_1\lrcorner F(y))\wedge\Imm\left(\mrm{e}^{-\ii\hat{\vartheta}}(\omega+\ii\,F(y))^{n-1}\right)=0
\end{gathered}
\end{equation*}
from which we finally obtain
\begin{equation*}
\begin{split}
\left\langle\mu(y),\vartheta([\zeta_1,\zeta_2])\right\rangle-\Omega(\hat{\zeta}_1,\hat{\zeta}_2)=&-\int\omega(V_2,V_1)\,\Rea\left(\mrm{e}^{-\ii\hat{\vartheta}}(\omega+\ii\,F(y))^n\right)
\end{split}
\end{equation*}
which is the same as~$\left\langle\mu(y),\vartheta([\zeta_1,\zeta_2])\right\rangle-\Omega(\hat{\zeta}_1,\hat{\zeta}_2)=-\langle\sigma,p([\zeta_1,\zeta_2])\rangle$.
\end{proof}
We can couple~$\tilde{\mu}$ to a second Hamiltonian action, as in~\cite{SchlitzerStoppa}. The algebra~$\f{h}$ acts in a Hamiltonian fashion on the space~$\m{J}$ of complex structures compatible with~$\omega$, so we can consider the induced diagonal action of~$\tilde{\f{g}}$ on~$\m{A}\times\m{J}$, obtaining a moment map~$\bm{\mu}:\m{A}\times\m{J}\to\tilde{\f{g}}^*$. Corresponding to the decomposition of~$\tilde{\f{g}}$ as a sum of~$\f{g}$ and~$\f{h}$, the equation~$\bm{\mu}(y,J)=0$ decomposes as
\begin{equation}\label{eq:dHYM_scal_momentmap}
\begin{dcases}
\Imm\left(\mrm{e}^{-\ii\hat{\vartheta}}\left(\omega+\ii(\beta+\mrm{d}y)\right)^n\right)=0\\
s(\omega,J)\,\omega^n-\lambda\,\Rea\left(\mrm{e}^{-\ii\hat{\vartheta}}\left(\omega+\ii(\beta+\mrm{d}y)\right)^n\right)=c\,\omega^n
\end{dcases}
\end{equation}
where~$\lambda$ is an arbitrary coupling constant, coming from the definition of symplectic form on~$\m{A}\times\m{J}$. The variables in~\eqref{eq:dHYM_scal_momentmap} are the complex structure~$J$ and the~$1$-form~$y$.

\subsubsection{Integrability}

We have shown that it is possible to interpret the system~\eqref{dKYM} as a moment map equation for the (infinitesimal) action of the Lie algebra~$\tilde{\f{g}}$ on~$\m{A}$. This is quite different from the theory of symplectic reductions, where one usually considers actions of Lie groups. However, as the bracket of~$\tilde{\f{g}}$ is ``twisted'' by~$\beta$, it is not clear how~\eqref{eq:inf_action_ext} could be realized as the action of a group of diffeomorphisms of~$X$.

Still, it is possible to define what the orbits of the action of this group \emph{should} be, as the distribution~$\m{D}\subset T\m{A}$ defined by~$
\m{D}_{y}:=\big\{\hat{\psi}_y\,\big\vert\,\psi\in\tilde{\f{g}}\big\}$ is integrable. We have already seen that~$\m{D}$ is \emph{formally} integrable in Lemma~\ref{lem:Lie_action}, but as~$\m{A}$ is an infinite-dimensional manifold this might not be sufficient to conclude that there is a foliation of~$\m{A}$ integrating~$\m{D}$. We can however exhibit an explicit parametrization for the leaves of~$\m{D}$; these leaves then play the role of the orbits of the action.
\begin{lemma}\label{lemma:integrabile}
The distribution~$\m{D}$ is integrable.
\end{lemma}
\begin{proof}
For a fixed~$y\in\m{A}$, we will exhibit a parametrization of the integral leaf of~$\m{D}$ through~$y$. We work under the simplifying assumption that~$H^1(X)=0$.

Let~$\m{H}$ be the group of Hamiltonian symplectomorphisms of~$(X,\omega)$, and consider
\begin{equation*}
\m{Y}:=\set*{(\Phi,\eta)\in\m{H}\times\m{A}^1(X)\tc\Phi^*F(y)-\beta=\mrm{d}\eta}.
\end{equation*}
We claim that the integral leaf of~$\m{D}$ through~$y$ is the image of the map~$Q:\m{Y}\to\m{A}$ defined by $Q(\Phi,\eta):=\Phi^*y+\eta$. As~$\mrm{Lie}(\m{H})=\f{h}$, the tangent space of~$\m{Y}$ at~$(\Phi,\eta)$ consist of pairs~$(V,\dot{\eta})\in\f{h}\times\m{A}^1(X)$ such that~$\dot{\eta}=V\lrcorner\Phi^*F(y)+\mrm{d}f$ for some function~$f$. The differential of~$Q$ then is
\begin{equation*}
DQ_{(\Phi,\eta)}(V,\dot{\eta})=\dot{\eta}+\partial_{t=0}\left(\Phi_{V,t}^*\Phi^*y\right)=\mrm{d}f+V\lrcorner\Phi^*F(y)+\m{L}_V(\Phi^*y).
\end{equation*}
By definition~$\Phi^*F(y)=\beta+\mrm{d}\eta$, so we obtain
\begin{equation*}
DQ_{(\Phi,\eta)}(V,\dot{\eta})=\mrm{d}f+V\lrcorner\beta+\m{L}_V\eta-\mrm{d}(\eta(V))+\m{L}_V(\Phi^*y)
\end{equation*}
which is the infinitesimal action at~$Q(\Phi,\eta)$ of~$(f-\eta(V),V)\in\tilde{\f{g}}$. By choosing different~$(V,\dot{\eta})\in T_{(\Phi,\eta)}\m{H}$, we can check that the differential of~$Q$ is surjective on~$\m{D}$.
\end{proof}

\subsubsection{Complexification}

We would like to consider~\eqref{eq:dHYM_scal_momentmap} as an equation for~$\omega$ and~$\beta$ in prescribed cohomology classes, keeping the complex structure of~$X$ fixed. Following~\cite{GarciaFernandez_PHD} (after~\cite{Donaldson_SymmKahlerHam}), we can formalize this shift in perspective by considering system~\eqref{eq:dHYM_scal_momentmap} along the \emph{complexified orbit} of a point~$(y,J)\in\m{A}\times\m{J}$ under the action~$\tilde{\f{g}}$. The idea is that the~$\f{h}$-complexified orbit of~$J$ is parametrized by the K\"ahler class of~$\omega$ (see~\cite{Donaldson_SymmKahlerHam}), while the~$\f{g}$-complexified orbit of~$y\in\m{A}$ parametrizes the~$(1,1)$-class of~$F(y)$.

To complexify the action~$\tilde{\f{g}}\curvearrowright\m{A}\times\m{J}$ we should however restrict attention to the set~$\m{P}\subset\m{A}\times\m{J}$ of pairs~$(y,J)\in\m{A}\times\m{J}$ such that~$F(y)$ is of type~$(1,1)$ with respect to~$J$. There is an integrable complex structure~$\bb{J}$ on~$\m{P}$ (see~\cite[Proposition~$2.2$]{AlvarezGarciaGarcia_KYM}), given by
\begin{equation*}
\bb{J}_{y,J}(k,A):=\left(-k\circ J, JA\right).
\end{equation*}
It is easy to check that the diagonal action of~$\tilde{\f{g}}$ on~$\m{A}\times\m{J}$, i.e.
\begin{equation*}
\widehat{(f,V)}_{(y,J)}=\left(\mrm{d}f+\m{L}_Vy+V\lrcorner\beta,\m{L}_VJ\right)
\end{equation*}
preserves~$\m{P}$, and the action is holomorphic: for every~$\zeta\in\tilde{\f{g}}$,~$\m{L}_{\hat{\zeta}}\bb{J}=0$. It is then possible to define on~$\m{P}$ an infinitesimal action of~$\tilde{\f{g}}^{\bb{C}}:=\f{\tilde{g}}\otimes_{\bb{R}}\bb{C}$, extending the action of~$\tilde{\f{g}}$. For any element~$\zeta=(f_1,V_1)+\ii (f_2,V_2)\in\tilde{\f{g}}^{\bb{C}}$, with~$f_i\in\m{C}^\infty(X)$ and~$V_i\in\f{h}$, we define its action as
\begin{equation*}
\begin{split}
\hat{\zeta}_{(y,J)}=&\widehat{(f_1,V_1)}_{(y,J)}+\bb{J}_{(y,J)}\widehat{(f_2,V_2)}_{(y,J)}=\\
=&\Big(\mrm{d}(f_1+y(V_1))+\mrm{d}^c(f_2+y(V_2))+(V_1+JV_2)\lrcorner F(y),\m{L}_{V_1+JV_2}J\Big).
\end{split}
\end{equation*}
Similarly to the action~$\tilde{\f{g}}\curvearrowright\m{A}$, also the action~$\tilde{\f{g}}^{\bb{C}}\curvearrowright\m{P}$ may not be lifted to a group action. However, it is possible to integrate the distribution~$\widehat{\m{D}}:=\m{D}+\bb{J}\m{D}$. The leaves of~$\widehat{\m{D}}$ can then be considered to be the orbits for the action of~$\tilde{\f{g}}^{\bb{C}}$ on~$\m{P}$. Following the similar discussion in~\cite{AlvarezGarciaGarcia_KYM}, we will show that the moment map equation along the complexified orbit of~$(y,J)\in\m{P}$ is equivalent to the following system of equations for two functions~$\varphi,\psi\in\m{C}^\infty(X)$
\begin{equation}\label{eq:dHYM_scal_complex}
\begin{dcases}
\Imm\left(\mrm{e}^{-\ii\hat{\vartheta}}\left(\omega+\mrm{d}\mrm{d}_J^c\varphi+\ii(F(y)+\mrm{d}\mrm{d}_J^c\psi)\right)^n\right)=0\\
s(\omega+\mrm{d}\mrm{d}_J^c\varphi)-\lambda\frac{\Rea\left(\mrm{e}^{-\ii\hat{\vartheta}}\left(\omega+\mrm{d}\mrm{d}_J^c\varphi+\ii(F(y)+\mrm{d}\mrm{d}_J^c\psi)\right)^n\right)}{(\omega+\mrm{d}\mrm{d}_J^c\varphi)^n}=c.
\end{dcases}
\end{equation}
Notice that in~\eqref{eq:dHYM_scal_complex}, the complex structure~$J$ and the~$1$-form~$y$ are fixed, so that~\eqref{eq:dHYM_scal_complex} is in fact a system of equations for a K\"ahler form and a~$2$-form belonging to the fixed classes~$[\omega]$ and~$[\beta]=[F(y)]$, respectively. In other words,
\begin{prop}\label{prop:complexification_momentmap}
The moment map equation~\eqref{eq:dHYM_scal_momentmap}, along the orbit of $(y,J)\in\m{P}$ under the complexified action~$\tilde{\f{g}}^{\bb{C}}\curvearrowright\m{P}$, is equivalent to~\eqref{modelEqIntro} for $B=F(y)$.
\end{prop}
\begin{rmk}
The $1$-form $y$ plays essentially no role in the complexification of \eqref{eq:dHYM_scal_momentmap}, as the Dolbeault cohomology class of $B=F(y)$ does not depend on $y$. The moment map equations obtained for different choices of $y$ are all equivalent (as long as the complex structure is fixed). In particular, if $\beta$ itself is a $(1,1)$-form, equation \eqref{modelEqIntro} is equivalent to the moment map equation along the complexified orbit of $(0,J)\in\m{P}$ (for $\beta=B$).
\end{rmk}
\begin{proof}[Proof of Proposition \ref{prop:complexification_momentmap}]
We start by exhibiting a parametrization of~$\widehat{\m{D}}$, obtained by modifying a construction in~\cite{Donaldson_SymmKahlerHam}. As in the proof of Lemma \ref{lemma:integrabile}, we work under the simplifying assumption~$H^1(X)=0$. Let~$K(\omega)$ denote the K\"ahler class of~$\omega$ defined by~$J$, and consider the space
\begin{equation*}
\tilde{\m{Y}}=\set*{\left(\Phi,\tilde{\omega},\eta\right)\in\mrm{Diff}_0(X)\times K(\omega)\times\m{A}^1(X)\tc\Phi^*\tilde{\omega}=\omega\mbox{ and }\Phi^*F(y)-\beta=\mrm{d}\eta}.
\end{equation*}
We claim that the integral leaf of~$\widehat{\m{D}}$ through~$(y,J)\in\m{P}$ is given by the image of
\begin{equation*}
\begin{split}
Q:\m{C}^\infty(X)\times\tilde{\m{Y}}&\to\m{A}\times\m{J}\\
\left(f,\Phi,\tilde{\omega},\eta\right)&\mapsto\left(\mrm{d}^c_{\Phi^*J}f+\eta,\Phi^*J\right).
\end{split}
\end{equation*}
The definition of $\tilde{\m{Y}}$ guarantees that the image of~$Q$ lies in~$\m{P}$. To prove that the image of~$Q$ is an integral leaf of~$\widehat{D}$, we show that the differential of~$Q$ is surjective on the ``purely imaginary'' part of~$\widehat{D}$, i.e.~$\bb{J}\m{D}$. This, together with the proof of Lemma~\ref{lemma:integrabile}, is sufficient to show that the tangent space to~$\mrm{ran}(Q)$ is~$\widehat{D}$.

Let~$\omega_t:=\omega-\mrm{d}\mrm{d}^c_Jt\varphi$ be a K\"ahler form with respect to the complex structure~$J$, and consider the time-dependent vector field~$V_t=JX^{\omega_t}_\varphi$. If~$\Phi_t$ is the isotopy of~$V_t$, then it is readily checked that~$\Phi_t^*\omega_t=\omega$. Moreover, if~$\eta_t$ is a path of~$1$-forms defined by
\begin{equation*}
\eta_0=y,\quad\dot{\eta}_t=\Phi_t^*\left(V_t\lrcorner F(y)\right)
\end{equation*}
then~$(\Phi_t,\omega_t,\eta_t)\in\tilde{\m{Y}}$ for every~$t$. Let also~$\set*{f_t\tc t\in\bb{R}},\set*{g_t\tc t\in\bb{R}}\subset\m{C}^\infty(X)$ be arbitrary smooth paths of functions. We compute the differential of~$Q$ along the path
\begin{equation*}
p_t:=\left(f_t,\Phi_t,\omega_t,\eta_t+\mrm{d}g_t\right)\in\m{C}^\infty(X)\times\tilde{\m{Y}}.
\end{equation*}
It will also be convenient to set~$J_t:=\Phi^*J$ and~$X_t:=\Phi_t^*X^{\omega_t}_\varphi=X^\omega_{\varphi\circ\Phi_t}$. Then
\begin{equation*}
Q\left(p_t\right)=\left(\mrm{d}^c_{J_t}f_t+\eta_t+\mrm{d}g_t,J_t\right).
\end{equation*}
We separately compute the derivatives of the single pieces. First notice that, as every~$J_t$ is integrable, the derivative in~$t$ of~$\Phi_t^*J$ is simply~$J_t\m{L}_{X_t}J_t$. For the first component of~$Q$ instead we have
\begin{equation*}
\begin{split}
\partial_t\left(\mrm{d}^c_{J_t}f_t\right)=&\mrm{d}^c_{J_t}\dot{f}_t+\left(\mrm{d}^c_{J_t}f_t\right)\circ\m{L}_{X_t}J_t=\\
=&\mrm{d}^c_{J_t}\dot{f}_t+\mrm{d}\left(X_t(f_t)\right)+(J_tX_t)\lrcorner\mrm{d}\mrm{d}^c_{J_t}f_t+\mrm{d}^c_{J_t}\left(X_t\lrcorner\mrm{d}^c_{J_t}f_t\right)
\end{split}
\end{equation*}
and by definition
\begin{equation*}
\partial_t\left(\eta_t+\mrm{d}g_t\right)=\Phi_t^*\left(V_t\lrcorner F(y)\right)+\mrm{d}\dot{g}_t=(J_tX_t)\lrcorner(\beta+\mrm{d}\eta_t)+\mrm{d}\dot{g}_t.
\end{equation*}
Putting all together, we find that~$Q_1:=\pi_1\circ Q$ satisfies
\begin{equation*}
\begin{split}
\partial_tQ_1(p_t)=&\mrm{d}^c_{J_t}\dot{f}_t+\mrm{d}\left(X_t(f_t)\right)+\mrm{d}^c_{J_t}\left(X_t\lrcorner\mrm{d}^c_{J_t}f_t\right)
+(J_tX_t)\lrcorner F(Q_1(p_t))+\mrm{d}\dot{g}_t=\\
=&\mrm{d}\left(X_t(f_t)\right)+\mrm{d}\dot{g}_t+
\mrm{d}^c_{J_t}\left(\dot{f}_t-\eta_t(X_t)-X_t(g_t)\right)+\\
&+\mrm{d}^c_{J_t}\left(X_t\lrcorner Q_1(p_t)\right)+(J_tX_t)\lrcorner F(Q_1(p_t)).
\end{split}
\end{equation*}
If we choose~$g_t$ so that~$X_t(f_t)+\dot{g}_t$ is constant, then we see that
\begin{equation*}
\partial_tQ(p_t)=\left(\mrm{d}^c_{J_t}\left(\dot{f}_t-X_t\lrcorner(\eta_t+\mrm{d}g_t)\right)+\mrm{d}^c_{J_t}\left(X_t\lrcorner Q_1(p_t)\right)
+(J_tX_t)\lrcorner F(Q_1(p_t)),J_t\m{L}_{X_t}J_t\right)
\end{equation*}
which is the infinitesimal action of~$\ii\!\left(\dot{f}_t-X_t\lrcorner(\eta_t+\mrm{d}g_t),X_t\right)\in\ii\tilde{\f{g}}$ at the point~$Q(p_t)$. Choosing different functions~$f_t$ and~$\varphi$, this construction can be used to produce paths that cover all the distribution~$\bb{J}\m{D}$.

To conclude the proof, we compose the moment map~$\bm{\mu}:\m{A}\times\m{J}\to\tilde{\f{g}}^*$ with~$Q$. Fix~$(\Phi,\tilde{\omega},\eta)\in\tilde{\m{Y}}$ and~$f\in\m{C}^\infty(X)$, and consider the equation~$\bm{\mu}\circ Q(f,\Phi,\tilde{\omega},\eta)=0$. Under the decomposition of~$\tilde{\f{g}}$ as~$\f{g}\oplus\f{h}$, it is equivalent to
\begin{equation*}
s(\omega,\Phi^*J) + \gamma \frac{\left(\omega+\ii F(\mrm{d}^c_{\Phi^*J}f+\eta)\right)^n}{\omega^n} = c.
\end{equation*}
By definition of~$\eta$ and~$\Phi$ however, we can rewrite this as
\begin{equation}\label{eq:complexified_momentmap_proof}
s(\Phi^*\tilde{\omega},\Phi^*J) + \gamma \frac{\left(\Phi^*\tilde{\omega}+\ii\left(\Phi^*F(y)+\mrm{d}\mrm{d}^c_{\Phi^*J}f\right)\right)^n}{\Phi^*\tilde{\omega}^n} = c.
\end{equation}
As~$s(\Phi^*\tilde{\omega},\Phi^*J)=\Phi^*s(\tilde{\omega},J)$ and~$c$ is a constant,~\eqref{eq:complexified_momentmap_proof} is equivalent to
\begin{equation*}
s(\tilde{\omega},J) + \gamma \frac{\left(\tilde{\omega}+\ii\left(F(y)+\mrm{d}\mrm{d}^cf\right)\right)^n}{\tilde{\omega}^n} = c
\end{equation*}
which is just~\eqref{eq:dHYM_scal_complex}, as~$\tilde{\omega}=\omega+\mrm{d}\mrm{d}^c\varphi$ for some~$\varphi\in\m{C}(X)$.
\end{proof}

\subsection{Futaki invariant}\label{FutakiSec}

The moment map interpretation of equation \eqref{eq:dHYM_scal_complex} gives a natural description of a character of the Lie algebra of the stabilizer of a point $(y,J)\in\m{P}$. This generalizes the analogue of the classical Futaki character, introduced in~\cite[Section~$2.6$]{SchlitzerStoppa} in the case when~$[B]= c_1(L)$ for some line bundle~$L\to X$, to the case of general complexified K\"ahler classes.

The (Lie algebra) stabilizer $\tilde{\f{g}}_{(y,J)}$ of $(y,J)\in\m{P}$ under the action of $\tilde{\f{g}}$ is given by pairs $(f,V)$ of a real function and a Hamiltonian vector field such that
\begin{equation*}
\m{L}_VJ=0\mbox{ and }\mrm{d}\left(f+y(V)\right)+V\lrcorner F(y)=0.
\end{equation*}
In other words, $V$ is be a real holomorphic vector field with a purely imaginary holomorphy potential, and it also has a potential with respect to $F(y)$. The stabilizer $\tilde{\f{g}}^{\bb{C}}_{(y,J)}$ for the \emph{complexified} action (which in general properly contains $\tilde{\f{g}}_{(y,J)}\otimes_{\bb{R}}\bb{C}$) instead is
\begin{equation*}
\tilde{\f{g}}_{(y,J)}^{\bb{C}}=\set*{(f,V)\in\m{C}^\infty(X,\bb{C})\times\f{h}_0\tc V\lrcorner F(y)=-\bdiff(f+y(V))}.
\end{equation*}
Notice that for each $V\in\f{h}_0$, there exists $f\in\m{C}^\infty(X,\bb{C})$ such that $(f,V)\in\tilde{\f{g}}^{\bb{C}}_{(y,J)}$, and it is determined uniquely up to addition of constants.

Let $B=F(y)$, and for any holomorphic vector field $V\in\f{h}_0$ let $\varphi(V,\omega)$ and $\varphi(V,B)$ be complex functions defined by the conditions
\begin{equation*}
\begin{split}
V\lrcorner\omega=\bdiff\varphi(V,\omega),\quad V\lrcorner B=\bdiff\varphi(V,B),\\
\int_X\varphi(V,\omega)\omega^n=\int_X\varphi(V,B)\omega^n=0.
\end{split}
\end{equation*}
Then, the functional $\m{F}_{[\omega],[B]}:\f{h}_0\to\bb{C}$ defined as
\begin{equation}\label{eq:Futaki_character_Bfield}
\m{F}_{[\omega^{\C}]}(V)=\int_X\varphi(V,\omega)\Rea(\gamma(\omega+\ii B)^n)-\varphi(V,B)\Imm(\gamma(\omega+\ii B)^n)-\varphi(V,\omega)s(\omega)\omega^n
\end{equation}
does not depend on the choices of $\omega$ and $B$ in the respective classes.  On the one hand, this is a consequence of our moment map picture, at the level of Lie algebras, Theorem~\ref{momentMapThm}. Alternatively, just as for the classical Futaki character~$\m{F}_{[\omega]} = \m{F}_{[\omega^{\C}]}|_{|\gamma|=0}$, the moment map picture is not actually required to show that~$\m{F}_{[\omega^{\C}]}$ is independent of the choice of representative for~$[\omega^{\C}]$: in fact, a slight variant of Futaki's original argument in~\cite{Futaki_obstruction}, differentiating $\m{F}_{[\omega^{\C}]}$ along a path between different representatives for $[\omega^{\C}]$, is enough to show this.   

Assume now that for any K\"ahler metric~$\omega$ there is a (unique) $B$-field $B(\omega)$ (in the prescribed class) that solves the dHYM equation. In this case, the character \eqref{eq:Futaki_character_Bfield} simplifies to the functional $\m{F}_{[\omega^{\bb{C}}]}$ of Definition \ref{def:Futaki_dHYM=0}.

\addcontentsline{toc}{section}{References}
\bibliographystyle{abbrv}
\bibliography{complexified_classes}
\vskip.5cm
CIRGET, Universit\'e du Qu\'ebec \`a Montr\'eal, Case postale 8888, Succursale centre-ville
Montr\'eal (Qu\'ebec) H3C 3P8\\
scarpa.carlo@uqam.ca\\
 
\noindent SISSA, via Bonomea 265, 34136 Trieste, Italy\\
Institute for Geometry and Physics (IGAP), via Beirut 2, 34151 Trieste, Italy\\
jstoppa@sissa.it    
\end{document}